\documentclass{CSML}

\pdfoutput=1

\usepackage{lastpage}

\def\dOi{13(3:35)2017}
\lmcsheading%
{\dOi}
{1--\pageref{LastPage}}
{}
{}
{Jan.~13, 2014}
{Sep.~28, 2017}
{}

\usepackage{graphicx}
\usepackage{mathrsfs}
\usepackage{caption}
\usepackage{subcaption}
\usepackage{mathpartir}
\usepackage{amsmath}
\usepackage{amssymb}
\usepackage{cmll}
\usepackage[all]{xy}
\usepackage{tikz}
\usepackage[hidelinks]{hyperref}
\usepackage{amsthm}
\usepackage{stmaryrd}
\usepackage{todonotes}

\usepackage{color}
\definecolor{Red}{cmyk}{0,1,2,0}
\newcommand{\red}[1]{#1}

\input{macros.sty}

\title[Games and strategies as event structures]{Games and strategies as event structures}
\author[S.~Castellan]{Simon Castellan\rsuper a}
\address{{\lsuper{a,b}}Univ Lyon, CNRS, ENS de Lyon, UCB Lyon 1, LIP}
\email{simon.castellan@ens-lyon.fr}

\author[P.~Clairambault]{Pierre Clairambault\rsuper b}
\address{\vspace{-18 pt}}
\email{pierre.clairambault@ens-lyon.fr}
\author[S.~Rideau]{Silvain Rideau\rsuper c}
\address{{\lsuper c}Department of Mathematics, UC Berkeley}
\email{silvain.rideau@berkeley.edu}

\author[G.~Winskel]{Glynn Winskel\rsuper d}
\address{{\lsuper d}Computer Laboratory, University of Cambridge}
\email{glynn.winskel@cl.cam.ac.uk}

\thanks{We gratefully acknowledge the support of the French LABEX MILYON
(ANR-10-LABX-0070), and of the ERC Advanced Grant ECSYM.}

\begin{document}


\maketitle

\begin{abstract}
  In 2011, Rideau and Winskel introduced \emph{concurrent games and strategies as event structures}, generalizing
  prior work on causal formulations of games. In this paper we give a detailed, self-contained and
  slightly-updated account of the results of Rideau and Winskel: a notion of pre-strategy based on event
  structures; a characterisation of those pre-strategies (deemed \emph{strategies}) which are preserved by 
  composition with a copycat strategy; and the construction of a bicategory of these strategies.
  Furthermore, we prove that the corresponding category has a compact closed structure, and hence forms
  the basis for the semantics of concurrent higher-order computation.
\end{abstract}



\section{Introduction}
Games are ubiquitous.  They appear in 
many areas, such as economics, logic, 
and computer science. They
provide a valuable language in which one can model situations where the evolution of a 
system is determined by the choices of several agents. The agents are players 
performing moves
according to rules that model the situation at hand, and the evolution of the system follows from
the sequence of moves reflecting the decisions of the players. The outcome of the game might
be a payoff for each player, a successful refutation of a logical formula, a bug exposed 
in a program -- or, in some instances, we might just be interested in the play itself as a 
description of the evolution of a system. Although games can in general involve many players, they 
often (as in this paper) focus on two players: 
Player (Proponent, $\exists$lo\"ise, Verifier, \dots) and Opponent ($\forall$b\'elard, Spoiler, \dots) 
each one defending their interests while subject to attacks of the other.
{Two-player games are perfectly suited to the representation of \emph{open systems}: one player plays for the
system at hand, while the other is deemed \emph{external} and plays for the environment. Player
may be regarded as representing a \emph{team} of players, whose interaction yields the system
under study.}

In their 
traditional formulation, games are highly sequential: the behaviour of a game determines 
 a tree of which the nodes are the positions and the branches describe the different choices available to a player.
The interaction between the players
results in the selection of a potentially infinite branch of the game tree. Most of the time,
each position belongs to exactly one player and the other has to wait until a move is played.
Often, the game also obeys the condition of \emph{alternation} where players are
additionally required to play in turns.

Despite this sequential nature, one also would like to use games to represent situations that are concurrent
or distributed, \emph{e.g.} several systems running in parallel, possibly with synchronizations
or shared resources. Of course such concurrent applications of games exist, but it is 
worth pointing out that in the overwhelming majority of cases concurrency is represented indirectly via the
interleaving, or linearization, of atomic actions of the participants. Rather than using a notion of game that 
does justice to the distributed nature of the system, a tree-based, inherently sequential
representation is opted for, where a branch is a total ordering of the implicitly partially ordered evolution
of the system. In other words, concurrency is modelled by removing alternation, but the basic
tree-based 
understanding remains unquestioned. Of course, that representation has been useful and sufficiently  accurate to a 
large extent, and a significant and successful body of work follows from this choice.  But we
believe nonetheless that a more precise \emph{causal} representation 
is to be preferred.
Our reasons and a further discussion on this point can be found in Section \ref{sec:evstrat}.

However, causal representations of concurrent processes have a richer structure than trees, and 
require more elaborate tools to be dealt with properly.  It was not clear at first on what mathematical
formalism one should rely on for this endeavour.
The first causal foundations for concurrent games emerged in the late nineties in the 
game semantics community; due to Abramsky and Melli\`es \cite{DBLP:conf/lics/AbramskyM99}, they were used to
build a fully complete model of multiplicative additive linear logic (MALL). The idea was 
to switch from a tree to a \emph{domain} of positions, and formulate (deterministic) strategies as closure operators
on this domain. Later, Melli\`es and Mimram \cite{DBLP:conf/concur/MelliesM07} connected this position-based approach to a more traditional 
play-based formulation in the framework of asynchronous games -- in this setting (deterministic) strategies were manipulated
as traditional sets of plays, but with closure properties ensuring an underlying causal order between moves.
In parallel, Faggian and Piccollo \cite{DBLP:conf/tlca/FaggianP09} had developed a setting where the (deterministic) strategies were manipulated explicitly as
partial orders, rather than the partial order being recovered \emph{a posteriori}. Finally, in 2011 Rideau and
Winskel \cite{lics11} generalized all prior work by proposing a setting where (non-deterministic) strategies are described
as event structures, thus benefiting from a 
body of work on event structure models for concurrency.

The present paper aims to be a detailed and self-contained introduction to this latter formulation of concurrent games:
it covers details and extends the results of \cite{lics11}. In Section \ref{sec:evstrat}, we start with a gentle 
introduction to the basic ideas behind the representation of concurrent processes as event structures, with
an eye towards the application to games. In this setting, both games and ``pre-strategies'' playing on them
are event structures, with a pre-strategy being essentially an event structure labelled by moves of
the game. But pre-strategies, thought of as prototypical strategies for Player,
 are too expressive: they 
{may} impose unreasonable constraints on  Opponent, 
and can behave in ways that are not consistent with
their standing for interaction in an asynchronous distributed environment.
As an answer to this, strategies are introduced in Section \ref{sec:strategies} as the pre-strategies that are preserved under composition with an \emph{asynchronous forwarder}, formalized as a copycat strategy. This provides an 
adequately robust notion of strategy on an event structure, and a non-deterministic generalization of the earlier
notions of concurrent strategies mentioned above. We prove the main result of \cite{lics11}: that strategies are
exactly the pre-strategies obeying conditions called \emph{receptivity} and \emph{courtesy}.
The paper~\cite{lics11} also constructed a \emph{bicategory} of concurrent games and 
strategies between them, akin to Joyal's category of Conway games \cite{JoyalGazette}.
In Section \ref{sec:bicategory}, we give a detailed proof of that result. Finally
in Section \ref{sec:compact} we show that just as Joyal's category, our category is compact closed and can provide a basis
for games-based models of higher-order computation. In Section \ref{sec:conclusion}, we conclude.

\paragraph{Other related work.} Many other notions of games for concurrency have appeared in the literature. 

In the verification community, ``concurrent games'' \cite{cglics00,cgtcs07} refer to variations of Blackwell
games \cite{martin1998determinacy}: there is a
tree (or a graph) of positions. The game is played in rounds: at each round, both players select their
behaviour from a pool of possible actions. This selection is independent, and with no information on 
the other player's choice. The next position is decided as a function of both player's choices. 
In contrast 
to our setting, their focus is on enforcing the independence of the two players in each round, rather than 
describing a general concurrent computation. In particular, plays are still totally ordered. Games in event
structures are closer to the games played in Zielonka automata \cite{DBLP:conf/icalp/GenestGMW13}, which could be unfolded to event structures.
However, our focus is more on the unfoldings themselves, and on their compositional structure.

Through our focus on compositionality, we are very close to the notions of games for concurrency studied in
the semantics community \cite{DBLP:journals/entcs/Laird01,DBLP:journals/apal/GhicaM08}. Just as
{we do}, they form categories of games and strategies where concurrent processes
can be modelled. However, these models are based on interleavings rather than partial orders: rather than 
opting for a primitive representation of concurrency based on partial orders, they represent
the execution of a concurrent process 
via the non-deterministic schedulings of its possible actions.

Finally, in a different direction, let us cite the ``playgrounds'' of Hirschowitz \emph{et al} \cite{DBLP:journals/cuza/HirschowitzP12,DBLP:conf/calco/Hirschowitz13},
and the multi-token Geometry of Interaction of Dal Lago \emph{et al} \cite{DBLP:conf/csl/LagoFHY14}.
Both formalisms aim at providing a non interleaving-based representation of concurrent processes and
of their execution. They should both relate to our approach, in the sense that from their settings one could
extract an event structure, which is arguably more abstract and syntax-independent than the models used there.

\section{Event structures, games and pre-strategies}
\label{sec:evstrat}

In this section we introduce the basic notions underlying our development, from event structures
to pre-strategies 
represented by 
them.

\subsection{Events for concurrent and distributed systems}\label{sec:es}

\paragraph{Causality and independence.}
It is common to describe the evolution of a process or system by listing its \emph{events},
\emph{i.e.} the observable actions occurring through time. For instance, one could describe
an interaction with a coffee vending machine as a sequence:
\[
\coin \cdot \coffee
\]
that we call a \textbf{trace}, where $\coin$ represents the action of inserting a coin in the machine, and $\coffee$ represents
the action of getting a coffee. In fact, the input/output behaviour of the vending machine may be
modelled by the set:
\[
\mathrm{Coffee} = \{\epsilon, \coin, \coin\cdot \coffee\}
\]
where $\epsilon$ is the empty sequence (and with possibly more iterations of the interaction if
one is not interested in a one-use coffee vending machine). Nearby the coffee machine, there is
a tea machine modelled by:
\[
\mathrm{Tea} = \{\epsilon, \coin', \coin' \cdot \tea\}
\]
where we use $\coin'$ to distinguish it from $\coin$. 

The two machines may be interacted with in parallel -- one may for instance pay for a coffee, 
then, while waiting for the machine to deliver, also pay for a tea, and then obtain both. This behaviour
may be represented as $\coin \cdot \coin' \cdot \coffee \cdot \tea$. In fact, the system 
formed by both machines can be modelled as:

\[
\begin{array}{l}
\{ \epsilon, \coin, \coin \cdot \coin', \coin \cdot \coffee, \coin \cdot \coin' \cdot \coffee, \coin \cdot \coffee \cdot \coin',\\
\coin \cdot \coin' \cdot \tea, \coin \cdot \coin' \cdot \coffee \cdot \tea, \coin \cdot \coin' \cdot \tea \cdot \coffee,\\
\coin \cdot \coffee \cdot \coin' \cdot \tea, \coin', \coin'\cdot \coin, \coin' \cdot \tea, \coin' \cdot \coin \cdot \tea,\\
\coin' \cdot \tea \cdot \coin, \coin' \cdot \coin \cdot \tea \cdot \coffee, \coin'\cdot \coin \cdot \coffee,\\
\coin' \cdot \coin \cdot \coffee \cdot \tea, \coin' \cdot \tea \cdot \coin \cdot \coffee\}
\end{array}
\]
This follows the so-called \emph{interleaving-based} approach to modelling concurrent and parallel systems: that two
independent processes interacted with in parallel should behave as the set of interleavings of the traces of
the original processes. This approach proved 
powerful and versatile, and 
provides the basis for 
most developments on models of concurrency. 

However, it suffers from some drawbacks. To cite two of them: \emph{(1)} as should appear clearly in our example, this representation gets exponentially bigger than 
the original system -- this is the so-called \emph{state explosion problem}, which is a 
main challenge in interleaving-based
model-checking of concurrent systems, \emph{(2)} it is unreadable, and obfuscates the key information 
of which events \emph{depend} 
on which events. Instead of the large set of traces above, one would like to 
manage with 
only the \emph{partial order} generating it displayed
in Figure \ref{fig:par_comp}, 
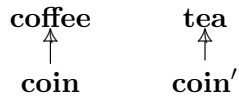
\begin{figure}[h!]
\[
\xymatrix@R=10pt{
\coffee&\tea\\
\coin	\ar@{-|>}[u]&
\coin'	\ar@{-|>}[u]
}
\]
\caption{Partial order semantics for the coffee and tea machines}
\label{fig:par_comp}
\end{figure}
for which the set of traces above is the set of all linearizations.
This idea is far from new: advocated first by Petri, it is
known as the \emph{independence}, \emph{partial order},  \emph{causal}, or \emph{truly concurrent}
approach to models of concurrency.
Although causal models yield smaller and more intuitive representations of the dynamics of a concurrent process, they can be quite subtle to manipulate and manage.
Operations that are straightforward for interleaving-based approaches, and say based on simple induction, can be more mathematically involved when carried out in a partial-ordered setting.

\paragraph{Event structures.}
Our example above is purely deterministic: it appears visibly in the partial order of Figure \ref{fig:par_comp} that 
no irreversible choice is ever made in the evolution of the system.
Whatever order the events of a prefix of the partial order of Figure \ref{fig:par_comp} appear, they 
can be completed to the maximal set $\{\coin, \coffee, \coin', \tea\}$. In this sense the order
in which these events occur is irrelevant. To express non-determinism, one needs to enrich the partial order. A natural way to do that is to
follow Winskel \cite{DBLP:conf/ac/Winskel86} and add a \emph{consistency} relation on top of the partial order, as follows. 
\begin{defi}[Event structures]
An \textbf{event structure} (\textbf{es} for short) is {a tuple} $(E, \leq_E, \Con_E)$ where $E$ is a set of events, $\leq_E$ is a
partial order on $E$ called \textbf{causality} and $\Con_E$ is a non-empty
set of finite subsets of $E$ called \textbf{consistency}, such that:
\[
\begin{array}{l}
\forall e\in E,~[e] = \{e'\in E\mid e' \leq_E e\}~\text{is finite},\\
\forall e\in E,~\{e\} \in \Con_E,\\
\forall X\in \Con_E,~\forall Y \subseteq X,~Y\in \Con_E,\\
\forall X\in \Con_E,~\forall e\in X,~\forall e'\leq_E e,~X\cup \{e'\} \in \Con_E
\end{array}
\]
We will often omit the subscripts in $\leq_E$, $\Con_E$ if they are obvious from the context{, and
use $E$ both for the event structure and its underlying set of events.}
\end{defi}

If $X\subseteq E$ is in $\Con$, then we say that it is
\textbf{consistent}, and its events 
may occur
together. The {states} of an event structure $E$, called
\textbf{configurations}, are the 
sets $x \subseteq E$ that are
both consistent (in the sense that every finite subset belongs to $\Con_E$) and \textbf{down-closed} (\emph{i.e.}  for all
$e\in x$, for all $e'\leq e$, one has $e'\in x$).
Here we shall work exclusively with finite configurations, those
  finite sets $x \subseteq E$ that are
both consistent and down-closed; the set of such
configurations of $E$ is written $\conf{E}$, and is partially ordered
by inclusion. Configurations with a maximal element are called
\textbf{prime configurations}, they are those of the form $[e]$ for
$e\in E$. We will also use the notation $[e) = [e] \setminus \{e\}$.
Between configurations, the \textbf{covering relation} $x \cov y$
means that $y$ is obtained from $x$ by adding exactly one event: $y$
is an \textbf{atomic extension} of $x$. We might also write
$x \longcov{e}$ to mean that $e \not \in x$ and
$x \cup \{e\} \in \conf{E}$; this says that the event $e$ is enabled at $x$.  
It is easy to prove that the relation $\subseteq$ between configurations is the transitive reflexive
closure of $\cov$; in fact $\cov$ is its transitive reduction.

When drawing event structures,
we will not 
portray the full partial order $\leq$ but its transitive reduction; the
\textbf{immediate causality} generating it defined as $e \imc e'$
whenever $e < e'$ and for any $e \leq e'' \leq e'$, either $e = e''$
or $e'' = e'$. Finally, we say two events $e, e'$ 
are 
\textbf{concurrent} when they are
consistent and incomparable for $\leq_E$.

Event structures can express non binary conflict
, \emph{e.g.} one can have three events $\{1, 2, 3\}$
{with consistent subsets those with at most} two elements: all events are
pairwise 
consistent, but not the three of them together. This extra generality makes for a smooth theory, 
but in many examples consistency is equivalently described by a complementary irreflexive binary \emph{conflict} relation $\conflict$, that
relates any two events that \emph{cannot} occur together, \emph{i.e.} $X \in \Con$ iff for all $e, e'\in X$, 
$\neg (e\conflict e')$. It follows then from the axioms of event structures that if $e \conflict e'$ and $e'\leq e''$,
{then} $e \conflict e''$ as well -- we call this conflict \textbf{inherited}. A conflict $e \conflict e'$ that is not inherited
is called \textbf{minimal}, and represented as $\xymatrix@C=10pt{e \ar@{~}[r]&e'}$. In order to alleviate the notation, when drawing
event structures with binary conflict we only represent minimal conflicts.

As an example, consider a (less popular) variant of the coffee machine above: when a coin is inserted it will produce 
a tea or a coffee, nondeterministically. The corresponding event structure can be represented as
{follows}:

\[
\xymatrix@R=10pt@C=10pt{
\coffee\ar@{~}[rr]&&\tea\\
&\coin
	\ar@{-|>}[ul]
	\ar@{-|>}[ur]
}
\]
Its configurations are $\{\{\emptyset\}, \{\coin\}, \{\coin, \coffee\}, \{\coin, \tea\}\}$.
We will never get \emph{both} tea and coffee even though both are enabled by $\coin$.

\paragraph{Simple parallel composition.} Whereas using traces the operation of putting two systems in parallel without communication
or interaction was the 
source of a combinatorial explosion, in event structures it only consists in putting
two event structures side by side. For instance, the event structure of Figure \ref{fig:par_comp} is obtained in a transparent way 
from event structures for the coffee and tea machines. 
Generally:
\begin{defi}
  Given two event structures $E$ and $F$ their \textbf{simple parallel composition} (or just \emph{parallel composition} for short)
  $E  \parallel  F$ is defined as the event structure comprising:
  \begin{itemize}
  \item \emph{Events:} $\{0\}  \times  E \cup \{1\}  \times  F$ (tagged disjoint union of $E$ and $F$),
  \item \emph{Causality}: $(i, c)  \leq _{E \parallel F} (j, c')$ when $i = j = 0$ and
    $c  \leq _E c'$ or $i = j = 1$ and $c  \leq _F c'$,
  \item \emph{Consistency} defined as:
    $$X  \in  \Con_{E  \parallel  F}\text{ iff } \{a \mid (0, a)  \in  X\}  \in  \Con_E\, \&\, \{b \mid (1, b)  \in  X\}  \in  \Con_F$$
  \end{itemize}
\end{defi}
Thus, $E \parallel F$ is $E$ and $F$ put side-by-side with no
causality or conflict between them. As a result, configurations of
$E \parallel F$ can be easily described in terms of those of $E$ and
$F$ -- namely there is a canonical order-isomorphism
$ \mathscr{C} (E \parallel F) \cong \mathscr{C} (E) \times \mathscr{C}
(F)$
(where configurations are ordered by inclusion). We will denote by
$x  \parallel  y  \in   \mathscr{C} (E  \parallel  F)$ the configuration corresponding to 
$(x, y)  \in   \mathscr{C} (E)  \times   \mathscr{C} (F)$.
When denoting events of a parallel composition $E_1 \parallel E_2$, we will not always 
write the explicit injections (as in $(0, e)$ or $(1, e)$). Instead, we will often
annotate or name the events so as to disambiguate the components they belong to (as in \emph{e.g.}
$e_1, e_2$).

\paragraph{Conjunctive causality and projection.}
In 
this setting of event structures causality
is \emph{conjunctive} rather than \emph{disjunctive}: states/configurations need to be down-closed, so for an 
event to occur it is required that \emph{all} of its dependencies have occurred before. For instance,
in the event structure of Figure \ref{fig:vendingmachine} the user needs to \emph{both} insert a coin and press a button in order
to get a drink (inserting a coin and pressing both buttons results in a non-deterministic choice).

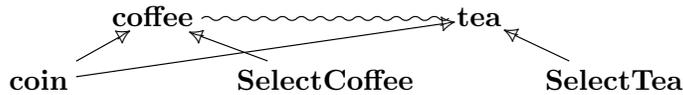
\begin{figure}[h!]
\[
\xymatrix@R=10pt@C=10pt{
&\coffee\ar@{~}[rr]&&\tea\\
\coin	\ar@{-|>}[ur]
	\ar@{-|>}[urrr]&&
\mathbf{SelectCoffee}
	\ar@{-|>}[ul]&&
\mathbf{SelectTea}
	\ar@{-|>}[ul]
}
\]
\caption{An event structure for a vending machine with selection}
\label{fig:vendingmachine}
\end{figure}

Plain event structures cannot express that an event may occur for two distinct, independent reasons -- such as
in saying that $\coffee$ can be obtain through a $\coin$ or through an override mechanism. In event structures, expressing that would require
two distinct events $\coffee$ and $\coffee'$, with different causal histories.
The apparent limitation that each event has a unique, unambiguous causal history enables us to perform the following
\emph{projection} operation\footnote{{
\emph{General event
structures}\cite{winskel-thesis} avoid this restriction, though they do not support a reasonable projection operation.}}:

\begin{defi}
If $E$ is an event structure and $V \subseteq E$ is a subset of events, then the \textbf{projection} $E \proj V$ has $V$ as events,
and causality and consistency directly inherited from $V$: if $e_1, e_2 \in V$ then $e_1 \leq_{E\proj V} e_2$ iff $e_1 \leq_E e_2$,
and for $X$ a finite subset of $V$, $X\in \Con_{E\proj V}$ iff $X \in \Con_E$.
\end{defi}

In other words, the projection $E\proj V$ is obtained by considering the events not in $V$ to be invisible: they occur silently,
and are not observable anymore. Because causality is conjunctive, for an event $e\in E\proj V$ there is never any ambiguity
as to what events caused it in $E$. 
Each configuration $y\in \conf{E}$ projects to $y\cap V \in \conf{E\proj V}$ -- reciprocally, any $x\in \conf{E\proj V}$
has a minimal \textbf{witness} $[x]_E = \{e'\in E \mid e'\leq_E e \in x\} \in \conf{E}$, yielding a bijection:
\[
\begin{array}{rcl}
\conf{E\proj V} &\iso& \{x \in \conf{E} \mid \forall e \in x\text{ maximal},~e\in V\}\\
x &\mapsto& [x]_E\\
y\cap V &\raisebox{5pt}{\text{\rotatebox{180}{$\mapsto$}}}& y
\end{array}
\]
that preserves and reflects inclusion.
This feature will be key 
to the \emph{hiding} step of the composition of strategies,
introduced later.

\paragraph{Polarity and pre-strategies.} We now move towards games. We consider \emph{two-player games} 
between Player (considered as having positive polarity) and Opponent (considered as having negative polarity).
Each event is equipped with a polarity, indicating which player has the responsibility to play it.

\begin{defi}
An \textbf{event structure with polarities} (esp for short) is an event structure $A$ along with a function
\[
\pol_A : A \to \{-, +\}
\]
associating to each event a polarity.
\end{defi}

When introducing events of an esp $A$, we might annotate them in order to indicate their polarity. For instance,
in ``let $a^- \in A$'', $a$ ranges over all events of $A$ of negative polarity. For configurations
$x, y \in \conf{A}$, we will write $x \subseteq^- y$ if $x\subseteq y$ and all events in $y\setminus
x$ are negative; $x \subseteq^+ y$ is defined dually. For {an esp} $A$,
we will write $A^\perp$ for its \textbf{dual}, \emph{i.e.} $A$ with the same data, except for the polarity which is
reversed. 

{ We define \textbf{games} to be simply esps. The terms
  \emph{game} and \emph{esp} will however not be used interchangeably:
  we will use \emph{games} for those esps used to specify the
  interface at which two players interact (\emph{strategies} will also
  be certain esps).}  For instance, one could model the interface of
the vending machine above by saying that Player plays 
according to the 
program of the coffee machine, Opponent plays for the user, and the
game describes the observable actions through which they interact on
the physical device.  Following this idea, the game for the physical
interface of the coffee machine would have events
$\{\coin^-, \mathbf{SelectCoffee}^-, \mathbf{SelectTea}^-, \coffee^+,
\tea^+\}$,
for causality the discrete partial order (\emph{i.e.} the order
contains only the reflexive pairs), and all sets consistent.  This
game is a discrete partial order, but in general games can feature
non-trivial causality and consistency.

The \emph{strategy} for Player would then describe the behaviour of the vending machine at this interface, represented as an event structure as well
(such as Figure \ref{fig:vendingmachine}). Both \emph{games} and \emph{strategies} are expressed as esps; they will nonetheless
play very different roles in the development. Following this idea, we now define \emph{pre-strategies} -- \emph{strategies},
defined later, will be subject to further conditions.

\begin{defi}
A \textbf{pre-strategy} on a game $A$ is an esp $S$ \emph{labelled} by $A$, that is, a function
$\sigma : S \to A$
which:
\begin{enumerate}
\item[(1)] Obeys the rules of the game (preserves configurations):
\[
\forall x \in \conf{S},~\sigma x \in \conf{A}
\]
\item[(2)] Plays linearly (local injectivity):
\[
\forall s, s' \in x \in \conf{S},~\sigma s = \sigma s' \implies s = s'
\]
\item[(3)] Preserves polarity: 
\[
\forall s\in S,~\pol_A(\sigma s) = \pol_S(s)
\]
\end{enumerate}
{Note that $\sigma$ does not need to preserve the order. On the other hand, as we will see later (Lemma
\ref{lemma:reflect}) it follows from these axioms that it always reflects it for consistent events.}
\end{defi}

As announced, a pre-strategy on $A$ is an esp $S$ along with a labelling function $\sigma : S \to A$. The esp structure
on $A$ brings constraints, that the labelling function has to respect.
It is easy to check that the event structure of Figure \ref{fig:vendingmachine} is a pre-strategy on
the game $M = \{\coin^-, \mathbf{SelectCoffee}^-, \mathbf{SelectTea}^-, \coffee^+, \tea^+\}$ (with
trivial causality and all subsets consistent), 
with the obvious map to it given by the labels. In the rest of this paper, when 
drawing pre-strategies we will follow the presentation of Figure \ref{fig:vendingmachine}: we will draw the event structure $S$, with
events written as their image 
via $\sigma$.

\red{
\begin{rem}
A pre-strategy $\sigma : S \to A$ is \emph{locally injective} but may not be \emph{injective}: there
may well be several incompatible events in $S$ mapping to the same event in $A$. This is true even
for two events $s_1, s_2 \in S$ sharing the same causal history, \emph{e.g.} $[s_1) = [s_2)$. For
instance, the following variant of Figure \ref{fig:vendingmachine} represents a valid pre-strategy on
the game $M$ above.
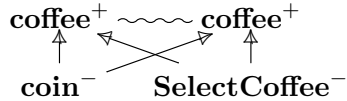
\[
\xymatrix@R=10pt@C=10pt{
\coffee^+\ar@{~}[r]&\coffee^+\\
\coin^-   \ar@{-|>}[u]
        \ar@{-|>}[ur]&
\mathbf{SelectCoffee}^-
        \ar@{-|>}[ul]
	\ar@{-|>}[u]
}
\]
The intuition here is that the coffee machine secretly tosses a coin (it has plenty of those after
all), but the result does not change its behaviour: it serves the coffee nonetheless. The 
intensional information of this nondeterministic choice, despite being unobservable, can be recorded
by our semantics. This intensionality means that our model will not validate the idempotency law for
nondeterministic choice: the event structures $\xymatrix{e\ar@{~}[r]&e}$ and $e$ are not isomorphic.

Though it is not covered in this paper, one may opt for a variant of our setting that does validate
idempotency (see \emph{e.g.} \cite{concur}).
But as a consequence we would also lose the branching information. To illustrate this,
consider the following variant of the pre-strategy above, playing on an extension of $M$ popular in
the United Kingdom.
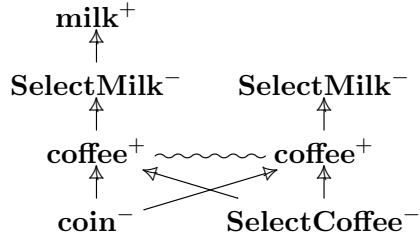
\[
\xymatrix@R=10pt@C=10pt{
\mathbf{milk}^+\\
\mathbf{SelectMilk}^-
	\ar@{-|>}[u]&
\mathbf{SelectMilk}^-\\
\coffee^+\ar@{~}[r]
	\ar@{-|>}[u]&\coffee^+
	\ar@{-|>}[u]\\
\coin^-   \ar@{-|>}[u]
        \ar@{-|>}[ur]&
\mathbf{SelectCoffee}^-
        \ar@{-|>}[ul]
        \ar@{-|>}[u]
}
\]
In this version, the user has the possibility of requesting milk after the coffee is served. But in
the pre-strategy above, they may not always get it: something may go wrong in the machine before the
coffee is served, leading to a state where $\mathbf{SelectMilk^-}$ is ineffective and the user
frustrated. The highly
intensional, non-idempotent representation of concurrency we opt for in this paper allows us to
record this branching information. In settings where this branching information is not relevant, such as
in \cite{concur}, it can be easily forgotten.
\end{rem}
}

{Besides being natural (we hope) as a first tentative definition of strategies, the pre-strategies
defined above match the standard notion of \emph{maps} between event structures.

\begin{defi}
If $E, F$ are event structures, a (total) \textbf{map of event structures} from $E$ to $F$ is a
function on events $f : E \to F$ satisfying \emph{(1)} and \emph{(2)} above.

The identity function on $E$ is a map of event structures and those are stable under composition; in other
words there is a category $\ES$ of event structures and maps.
\end{defi}

Later on we will also consider \emph{partial} maps between event structures (Definition
\ref{def:partialmap}), but throughout this paper all maps are considered total unless explicitly
said otherwise. We also have a category $\ESP$ of esps, and maps \emph{preserving polarities} --
technically pre-strategies are exactly maps of esps. We keep a distinguished terminology, because in
the sequel we will encounter maps of esps that it is unwise to regard as pre-strategies.}

We note in passing that simple parallel composition extends to esps by defining the polarity of
$A \parallel B$ as $\pol_{A \parallel B}(0, a) = \pol_A(a)$ and
$\pol_{A \parallel B}(1, b) = \pol_B(b)$ {-- this entails that parallel composition commutes
with the duality operation, \emph{i.e.} $(A\parallel B)^\perp = A^\perp \parallel B^\perp$}. 
Two pre-strategies $ \sigma  : S  \rightarrow  A$ and
$ \tau  : T  \rightarrow  B$ playing respectively on $A$ and $B$ can be combined to
form a pre-strategy $ \sigma   \parallel   \tau  : S  \parallel  T  \rightarrow  A  \parallel  B$ defined by
$( \sigma   \parallel   \tau )(0, a) = (0,  \sigma (a))$ and $( \sigma   \parallel   \tau )(1, b) = (1,  \tau (b))$. 
In fact with this definition, simple parallel composition acts functorially on maps of es and esp and
equip the categories $\ES$ and $\ESP$ with the structure of a symmetric monoidal category (with the 
empty event structure $1$ as unit).

At this point, the reader may find confusing the fact that although there are polarities in games
and pre-strategies, these are not taken into special account in the definition of pre-strategies
{(besides their preservation by the labeling function).
This makes pre-strategies more powerful than perhaps wished:
a pre-strategy may constrain the external Opponent beyond the rules of the game.
Taking $S = 1$ the empty event structure, and writing $\ominus$ for the esp with just one negative event, the empty
map
$S \to \ominus$
is a valid pre-strategy. As $S$ has no counterpart for the unique negative move in the game, this pre-strategy
fails to acknowledge Opponent's right to play it. This is in contradiction with the idea that 
Opponent's available actions should only depend on the game, and not be controllable by Player. We
will see in Section \ref{sec:strategies} other ways in which pre-strategies may constrain Opponent in
unintended ways.}

This is because the current definition is an intermediate step, towards the notion of strategy
introduced in Section \ref{sec:strategies}
that will take polarity more carefully into account. Whereas pre-strategies axiomatize the
polarity-agnostic description of the evolution of a concurrent process
on an interface, strategies will satisfy polarity-specific constraints, \emph{e.g.} a strategy
cannot prevent its opponent from playing a move 
enabled in the game. But for the remainder of this section, polarities 
are present only to set the stage for Section \ref{sec:strategies}.

\subsection{Interaction of pre-strategies}\label{sec:pb} Pre-strategies playing on
$A^\perp$ are pre-strategies for Opponent 
or \emph{counter pre-strategies}. Given a pre-strategy
$ \sigma : S \rightarrow A$ and a counter pre-strategy $ \tau : T \rightarrow A^\perp$,
we proceed to explain how they \emph{interact} with each other. {As $\sigma$ and $\tau$ have
opposite expectations for the polarity of events in $A$, their interaction will drop all information
about polarities, \emph{i.e.} it will be a \emph{map of event structures} (in $\ES$):
\[
\sigma \wedge \tau : S \wedge T \rightarrow A
\]
where $A$ is silently coerced to an event structure \emph{without polarities}. We regard it again as
an event structure $S\wedge T$ describing the causal structure of actions that both $\sigma$ and
$\tau$ agree to, along with a labeling $\sigma \wedge \tau$ of $S\wedge T$ by events of (a polarity-agnostic version of)
$A$.

In fact, the \emph{definition} of interaction makes use neither of the polarity information in $S, T$, or
$A$, nor of its preservation by $\sigma$ and $\tau$. So we will define it in this section as an operation
which for two maps $\sigma : S \to A$ and $\tau : T \to A$ in $\ES$, yields a map $\sigma
\wedge \tau : S \wedge T \to A$ in $\ES$. Polarities will only become relevant again in Section
\ref{sec:comp}, when we define composition.}


As we will see, interaction is very close to the product of event structures used
in \cite{icalp82,DBLP:conf/ac/Winskel86} to interpret the synchronising parallel composition of CCS (we will
see that it corresponds to a \emph{pullback} in $\ES$).

\paragraph{Secured bijections.}\label{sec:secbij}
The interaction of $ \sigma $ and $ \tau $ should follow the behaviour
that $\sigma$ and $\tau$ agree on: in a given state, it should be ready to play
$c \in A$ whenever $\sigma$ and $\tau$ are. In particular, this means
that an event $c \in A$ played by $ \sigma $ and $ \tau $ should be
played in their interaction only after all the dependencies in $S$ and
$T$ are satisfied. For instance the interaction of the following two
event structures labelled on the interface $A = a\ b\ c$ (consisting in
three concurrent events)

$$
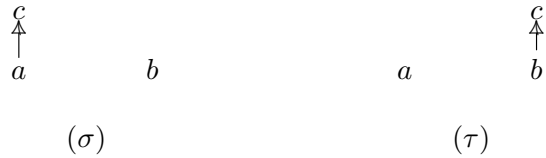
\xymatrix@C=0.3cm@R=10pt{
c & & & & & & & & & & c \\
a\ar@{-|>}[u] & & b & & & &  & & a & & b \ar@{-|>}[u] \\
& ( \sigma ) & & & & & & & & ( \tau )
}$$
should give rise to the interaction $ \sigma   \wedge   \tau $:

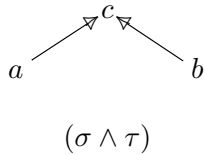
$$\xymatrix@C=0.3cm@R=10pt{
& c \\
a\ar@{-|>}[ur] & & b \ar@{-|>}[ul] \\
& ( \sigma   \wedge   \tau )
}$$
with immediate causal links imported from both $S$ and $T$. 
Similarly, a set of events should be consistent in the interaction when the
corresponding projections in $S$ and $T$ are.

At this point, one is tempted to define 
the events of
$S \wedge T$ as \emph{synchronized events}: pairs
$(s, t) \in S \times T$ such that $ \sigma s =\tau t$. This works
correctly when the maps $ \sigma $ and $ \tau $ are injective but fails in general.
For instance, consider the interaction of the two labelled event structures:

$$
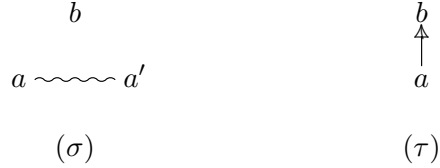
\xymatrix@C=0.15cm@R=10pt{
 & b & & & & & & & & & & b  \\
a \ar@{~}[rr] & & a' & & & & & & & & & a\ar@{-|>}[u] \\
& ( \sigma ) & & & & & & & & &  & ( \tau )
}$$

Here, $ \sigma $ has two copies $a$ and $a'$ of the event
$a \in A$ (by local injectivity, the two copies must be in conflict)
and $ \tau $ plays $b$ after $a$. However, because $ \sigma $ has two ways of
playing $a$, the interaction has two possible causal histories
for $b$: either after $(a, a) \in S \times T$ or after
$(a', a) \in S \times T$. Since in event structures, each event
comes with a unique causal history, those two histories for $b$ must correspond to
\emph{two different events} in $S \wedge T$, which should therefore
look like:

$$
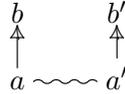
\xymatrix@R=10pt{
b & b'\\
\ar@{-|>}[u]a \ar@{~}[r] & a'\ar@{-|>}[u] \\
}$$

We see that $S \wedge T$ has four events, whereas there are only
three possible synchronized pairs: $(a, a)$, $(a', a)$ and $(b, b)$ --
thus events of $S \wedge T$ will be more than just pairs.

{Our approach will be to construct the desired event structure $S \wedge T$
indirectly via the set of configurations that we wish it to have.
In the example above, configurations of the diagram} \emph{are} in one-to-one
correspondence with synchronized configurations: pairs
$(x, y) \in \mathscr{C} (S) \times \mathscr{C} (T)$ such that
$ \sigma x = \tau y$. By local injectivity, in such a situation
$ \sigma $ and $ \tau $ induce a bijection
$ \varphi : x \simeq \sigma x = \tau y \simeq y$ that is not
order preserving in general (we use the notation $\simeq$, as opposed to $\cong$,
to insist on the fact that although $x, \sigma\,x, \tau\,y$ and $y$ are canonically partially
ordered by $\leq_S, \leq_A, \leq_T$, these bijections do not preserve this order).
Note that its graph is a set of synchronized (paired) events as above. 

Such bijections will be 
used to represent \emph{configurations} of the interaction. But
as configurations of an event structure (yet to be defined), the
graphs of these bijections should be ordered as well.
As shown above, the order on
$S \wedge T$ should be inherited from that of $S$ and $T$. However,
the transitive closure of the relation induced by the orders of $S$ and
$T$ is, in general, not an order. For instance in the following picture

$$
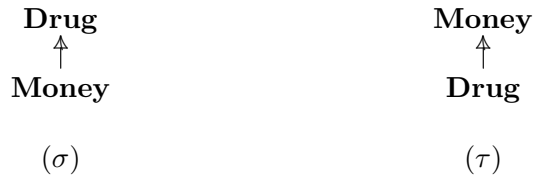
\xymatrix@R=10pt{
\mathbf{Drug} & & & & \mathbf{Money} \\
 \ar@{-|>}[u] \mathbf{Money} & & & & \mathbf{Drug}\ar@{-|>}[u]  \\
( \sigma ) & & & & ( \tau )
}$$
there is a \emph{deadlock}: $ \sigma $ (the dealer) waits for the
money to be delivered before 
presenting the drug 
while $ \tau $ (the buyer)
waits for the drug before 
offering the dollars. Their interaction
should be empty as in the empty configuration there is no common event
that $ \sigma $ and $ \tau $ are both ready to play. This is reflected
by the fact that on the bijection
$\{(\mathbf{Money}, \mathbf{Money}), (\mathbf{Drug}, \mathbf{Drug})\}$
the preorder induced by $S$ and $T$ is not an order: it has a loop. To
eliminate such loops, we introduce \emph{secured bijections}:

\begin{defi}[Secured bijection]
  A \textbf{secured bijection} between two (finite) orders $(q,  \leq _q)$ and $(q',  \leq _{q'})$ is
  a bijection $ \varphi  : q  \simeq q'$ such that
  the reflexive and transitive closure of the following relation on the graph of $ \varphi $ is an order:
  $$(a, b) \vartriangleleft (a', b')\quad\text{when}\quad a <_q a'\text{ or }b <_{q'} b'$$
\end{defi}

Secured bijections need not preserve the order but they do not
contradict it: if $a <_q b$ then $ \varphi\,b \not <_{q'}  \varphi\,a$ as otherwise this would
constitute a cycle. 

{
Equivalently, secured bijections are those bijections satisfying a reachability property akin to one of
configurations of event structures, which can always be reached from the empty configuration by
successive additions of events. We invite the reader to check the following lemma, which is useful
in forging an intuition on the role of the notion.

\begin{lem}
Let $(q, \leq_q)$ and $(q', \leq_{q'})$. Then a bijection $\varphi : q \simeq q'$ is secured, iff
there is a sequence of (graphs of) bijections:
\[
(\varphi_0 : x_0 \simeq y_0) \longcov{(a_1, b_1)} (\varphi_1 : x_1 \simeq y_1) \longcov{(a_2, b_2)}
\dots \longcov{(a_n,b_n)} (\varphi_n : x_n \simeq y_n)
\]
such that $\varphi_0$ is the empty bijection, $\varphi_n = \varphi$, and for all $0 \leq i \leq
n$, $x_i \in \conf{q}$ and $y_i \in \conf{q'}$ (\emph{i.e.} they are down-closed).
\end{lem}
}


Secured bijections can be used to give a very concise description of
the desired states of $S \wedge T$: write $ \secbij{\sigma, \tau}$ for the following set, ordered by inclusion.
$$\secbij{\sigma, \tau} = \{ \varphi \mid  \varphi  : x  \stackrel{ \sigma } {\simeq}   \sigma x =  \tau y  \stackrel  \tau   \simeq   y\text{ is secured, with $x  \in   \mathscr{C} (S), y  \in   \mathscr{C} (T)$} \}.$$

Since secured bijections are by definition equipped with a canonical
order, the elements of $\secbij{ \sigma , \tau }$ can be seen as
ordered sets. Immediate causal links in a secured bijection are
related to those of the underlying orders:
\begin{lem}
  \label{lemma:imc_bij} Let $ \varphi  : q  \simeq  q'$ be a secured bijection. If we
  have $(a, b)  \rightarrowtriangle _{ \varphi } (a', b')$ then either $a \rightarrowtriangle _q a'$ or $b \rightarrowtriangle _{q'} b'$.
\end{lem}
\begin{proof}
  From $(a, b)  \rightarrowtriangle _{ \varphi } (a', b')$ we deduce
  $(a, b) \vartriangleleft (a', b')$. Hence either $a <_q a'$ or
  $b <_{q'} b'$. Assume for instance $a <_q a'$. If we do not have
  $a \imc_q a'$ then there exists $a_0  \in  q$ such that $a < a_0 <
  a'$.
  Then $(a_0,  \varphi\,a_0)  \in   \varphi $ and we have
  $(a, b) <_ \varphi  (a_0,  \varphi\,a_0) <_{ \varphi } (a', b')$ contradicting the
  hypothesis.
\end{proof}
\paragraph{Prime secured bijections.}
The order $(\secbij{\sigma, \tau}, \subseteq)$ is (up to isomorphism)
the order of configurations of the event structure we are looking
for. We can now reconstruct an event structure whose order of
configurations matches this order: events are identified as the
\emph{prime} secured bijections, \emph{i.e.} those with a top (\emph{i.e.} greatest)
synchronized event $(s, t)$. In other words there will be an event for each synchronized pair $(s,
t)$ along with a consistent causal history for it, \emph{i.e.} a prime secured bijection with $(s,
t)$ as top element.
In
particular, if there is none (because of a cycle),
it would not appear in the interaction. 
With these ingredients we can form an event structure:

\begin{defi}[Interaction of pre-strategies]
  Let $ \sigma  : S  \rightarrow  A$ and $ \tau  : T  \rightarrow  A$ be maps of event structures. We
  define the event structure $S  \wedge  T$ as follows:
  \begin{itemize}
  \item \emph{Events:} those elements of $ \secbij{\sigma, \tau}$ that have a top event,
  \item \emph{Causality:} inclusion of graphs,
  \item \emph{Consistency:} a finite set $X$ of (graphs of) secured
    bijections is consistent when its union is still (the graph of) a
    secured bijection in $ \secbij{\sigma, \tau}$.
  \end{itemize}
\end{defi}

\noindent We invite the reader to apply this definition on the examples at the beginning of
Section \ref{sec:secbij}, and check that we obtain the event structures announced.

It is routine to check that $S \wedge T$ is an event structure such
that $ \mathscr{C} (S \wedge T)$ is order-isomorphic to
$\secbij{\sigma, \tau}$: 
\begin{lem}\label{lemma:confinter}
  For each configuration $x  \in   \mathscr{C} (S \wedge T)$, {then 
$ \varphi _x = \cup x : x_S  \simeq  x_T  \in  \secbij{ \sigma ,  \tau }$ is a secured bijection.
Moreover, this assignment is such that} 

  $$
  \begin{aligned}
     \varphi _x & \rightarrow  x \\
    (s, t) &\mapsto [(s, t)]_{ \varphi _x}
  \end{aligned}$$
  is an order-isomorphism $ \varphi _x  \cong  x$, where $[(s, t)]_{\varphi_{{x}}} $ denotes the down-closure of $(s, t)$ inside the ordered set
  $ \varphi _x$. Moreover, the mapping $x \mapsto \varphi _x$ defines an order
  isomorphism
  $ \mathscr{C} (S \wedge T) \cong \secbij{ \sigma , \tau }$.
\end{lem}
\begin{proof}
  Let $x \in \mathscr{C} (S \wedge T)$. By definition of consistency
  in $S \wedge T$, $\cup x$ is the graph of a secured bijection
  $ \varphi _x \in \secbij{ \sigma , \tau }$. Any element of $x$ is
  a secured bijection with a maximal element $(s, t)$, and hence 
  is $[(s, t)]_{\varphi_x}$.
  Thus, $[(s, t)]_{ \varphi _x} \mapsto (s, t)$ defines an
  order-isomorphism $x \cong \varphi _x$. This yields a map
  $ \mathscr{C} (S \wedge T) \rightarrow \secbij{ \sigma , \tau }$.
  The converse maps a secured bijection $ \varphi $ to the set of
  elements of $S \wedge T$ included in $ \varphi $.
\end{proof}

By local injectivity of $\sigma$ and $\tau$, a secured bijection 
$\varphi : x \simeq y$ is entirely determined by $x$ and $y$. Therefore, $\conf{S \wedge T}$ 
is also order-isomorphic to the set of pairs $(x, y) \in \conf{S} \times \conf{T}$ such that
$\sigma x = \tau y$ and such that the induced bijection between $x$ and $y$ is secured, 
partially ordered by componentwise inclusion -- we will use this description later on in the proofs.

\paragraph{The interaction pullback.}
Events of $S  \wedge  T$ have the form $\varphi$ with a top element
$(s, t)$. The mappings $ \Pi _1 :  \varphi \mapsto s$ and $ \Pi _2 :  \varphi \mapsto t$
induce maps of event structures $S  \wedge  T  \rightarrow  S$ and $S  \wedge  T  \rightarrow  T$ that make
the following diagram commute:
$$\xymatrix@R=10pt@C=10pt{
& S  \wedge  T \ar[dl]_{ \Pi _1}\ar[dr]^{ \Pi _2} \\
S \ar[dr]_{ \sigma } & & T \ar[dl]^{ \tau } \\
& A
}$$

Writing $\pi_i$ for the (set-theoretic) projections, by Lemma \ref{lemma:confinter}, for every $x  \in   \mathscr{C} (S \wedge T)$ we have
$$ \pi _1  \varphi _x =  \Pi _1 x$$ 
as $ \pi _1(s, t) = s =  \Pi _1[(s, t)]_{ \varphi _x}$ and similarly for $ \pi _2$ and $ \Pi _2$.  Those maps furthermore satisfy a
universal property making formal the intuition of a ``generalized
intersection'': $(S \wedge T, \Pi _1, \Pi _2)$ is the \emph{pullback}
of $ \sigma : S \rightarrow A$ and $ \tau : T \rightarrow A$, meaning
that that the above diagram commutes and for each map of event
structures $ \alpha : X \rightarrow S$ and $ \beta : X \rightarrow T$
satisfying

  $$
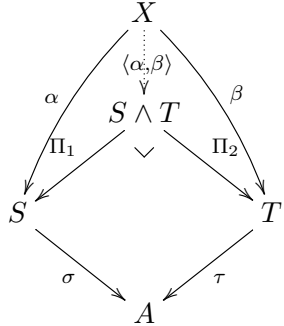
\xymatrix{
    & X \ar@{.>}[d]^{\hspace{-0.42cm} \langle  \alpha ,  \beta \rangle}\ar@/_/[ddl]_{ \alpha } \ar@/^/[ddr]^{ \beta } \\
    & S  \wedge  T \pb{270}\ar[dl]_{ \Pi _1} \ar[dr]^{ \Pi _2} \\
    \ar[dr]_{ \sigma } S & & T \ar[dl]^{ \tau } \\
    & A }$$
  there is a unique map $ \langle  \alpha ,  \beta  \rangle  : X  \rightarrow  S  \wedge  T$ such that
  $ \Pi _1 \circ  \langle  \alpha ,  \beta  \rangle  =  \alpha $ and $ \Pi _2 \circ  \langle  \alpha ,  \beta  \rangle  =  \beta $.

  To construct $ \langle \alpha , \beta \rangle $, we will need the
  following lemma stating 
  the precise sense in which maps of event structures reflect the
  causal order:
\begin{lem} \label{lemma:reflect} 
  Let $f : A  \rightarrow  B$ be a map of event structures and $a, b  \in  A$ such
  that $\{a, b\}$ is consistent. If $f(a)  \leq  f(b)$ then $a  \leq  b$.
\end{lem}
\begin{proof}
  Since $f$ is a map of event structures, $f[b]$ is down-closed as a
  configuration of $B$. Since $f(a) \leq f(b) \in f[b]$ by hypothesis, it
  follows that $f(a) \in f[b]$ and thus $f(a) = f(c)$ for some
  $c \leq b$.  Since $\{a, b\}$ is consistent so is $\{a, b, c\}$ and
  local injectivity implies $a = c \leq b$ as desired.
\end{proof}

We can now prove that our construction yields a pullback:
\begin{lem}[The interaction is a pullback]\label{lem:pullback}
  Let $ \sigma  : S  \rightarrow  A$ and $ \tau  : T  \rightarrow  A$ be maps of event structures. The
  triple $(S  \wedge  T,  \Pi _1,  \Pi _2)$ is a pullback for $ \sigma $ and $ \tau $.
\end{lem}

\begin{proof}
  We have already noticed that the inner square commutes.

  \emph{Existence of $ \langle \alpha , \beta \rangle $:} Assume we
  have an event structure $X$ with two maps
  $ \alpha : X \rightarrow S$ and $ \beta : X \rightarrow T$ such that
  $ \sigma \circ \alpha = \tau \circ \beta $. Let $a \in X$. The
  bijection (by local injectivity of $\alpha, \beta$):
\[
\varphi_a = \{(\alpha\,a', \beta\,a')\mid a'\leq_X a\} : \alpha [a] \simeq \beta [a]
\]
  is secured as a consequence
  of Lemma \ref{lemma:reflect}, as a cycle in it would be reflected to
  $X$. 
  Define
  $ \langle \alpha , \beta \rangle (a) = [ (\alpha (a), \beta
  (a))]_{\varphi_a}$
  to be the secured bijection obtained as the down-closure of
  $(\alpha(a), \beta(a))$ inside the canonical order on the graph of
  $\varphi_a$: it has a maximal event by
  construction, and thus is an event of $S \wedge T$.  It is a good exercise to check that 
  this function defines a map of event structures; which makes the two triangles
  commute.

  \emph{Uniqueness of $ \langle \alpha , \beta \rangle $:} Assume we have
  another map $ \psi : X \rightarrow S \wedge T$ making the two
  triangles commute. We will check that $\langle \alpha , \beta \rangle $ and $\psi$ have the same
action on configurations, which will imply (by Lemma \ref{lem:faithful} below) that they are the
same. Let $z  \in   \mathscr{C} (X)$. Its image through $ \psi $ and
  $ \langle  \alpha ,  \beta  \rangle $ are (under the order-isomorphism $ \mathscr{C} (S  \wedge  T)  \cong   \secbij{\sigma, \tau})$
  secured bijections $ \varphi : x \simeq y$ and $ \varphi' : x' \simeq y'$. But by commutation of
the two triangles in the pullback we must have $x = x' = \alpha\,z$ and $y = y' = \beta\,z$, thus
$\varphi = \varphi'$ (as $\varphi$ is uniquely determined from $x, y$ by local injectivity).  
\end{proof}

In the proof of uniqueness, we only compared the maps by their action on
configurations and deduced they were equal on
events. This is justified by the following simple fact, that will be useful later on:
\begin{lem}
  Let $f, g : A  \rightarrow  B$ be parallel maps of event structures such that
  for all configuration $x  \in   \mathscr{C} (A)$ we have $fx = gx$. Then $f = g$.
\label{lem:faithful}
\end{lem}
\begin{proof}
  Let $a  \in  A$. Write $[a)$ for the configuration
  $[a] \setminus \{a\}$. By hypothesis we have $f[a] = g[a]$ and
  $f[a) = g[a)$ as sets, thus
  $\{f(a)\} = f[a] \setminus f[a) = g[a] \setminus g[a) = \{g(a)\}$
  and hence $f(a) = g(a)$.
\end{proof}
\subsection{Composition of pre-strategies}\label{sec:comp}

Building on our understanding of the interaction of pre-strategies as a pullback,
we can now proceed to define the notion of composition, which is of critical
importance in particular for the application of our games to semantics of programming
languages. For that we need to define what is a pre-strategy $\sigma$ \emph{from game $A$ to
game $B$}, and given also $\tau$ from $B$ to $C$, what is $\tau \odot \sigma$ from $A$
to $C$.
%
%

Following Joyal \cite{JoyalGazette}, we will define a pre-strategy
from $A$ to $B$ to be simply a pre-strategy on the composite game
$A^\perp \parallel B$. Let us show how to compose such pre-strategies.
From
$ \sigma : S \rightarrow A^\perp \parallel B$ and
$ \tau : T \rightarrow B^\perp \parallel C$, we need to build a
pre-strategy $ \tau \odot \sigma $ on the game $A^\perp \parallel C$. Note that from such a notion
of composition we can recover a notion of application when $A$ is the empty event structure $1$.
As usual in game semantics, composition is defined in two steps: firstly, we construct the \emph{interaction}
of the two strategies as an event structure where the two strategies communicate freely. Secondly, the 
internal synchronisation steps are \emph{hidden} away. We will now detail these two steps.

To illustrate them, let
$\mathbb B$ be the game $\xymatrix@C=10pt{\ar@{~} [r]\ttrue^+& \ffalse^+}$ of booleans (two
conflicting positive events). Consider the following pre-strategies $ \sigma $
and $ \tau $ respectively playing on $1^\perp  \parallel  \mathbb B_1$ and $\mathbb B_1^\perp
\parallel  \mathbb B_2$ {(where indices are just there to disambiguate otherwise identical copies of
$\mathbb B$)}:

$$\xymatrix@R=10pt@C=20pt{
 & & & & &  \ffalse_2^+ & & \ttrue_2^+ \\
  \ttrue_1^+ \ar@{~}[rr] & & \ffalse_1^+ & & & \ttrue_1^- \ar@{-|>}[u] \ar@{~}[rr] & &  \ffalse_1^-\ar@{-|>}[u]\\
  & ( \sigma ) & & & & & ( \tau ) }$$

The pre-strategy $ \sigma $ performs a nondeterministic choice: it can either
play true or false. Likewise, $ \tau $ computes the negation of a boolean: when Opponent
plays true or false on $\mathbb B_1$ it answers the negation of that in
$\mathbb B_2$.

\paragraph{Interaction.} Ignoring the polarities, $ \sigma $ and $ \tau $
are maps of event structures $S \rightarrow A \parallel B$ and
$T \rightarrow B \parallel C$. They do not play on the same game
so it is not possible to make them interact directly. To solve this
problem we pad them out with identity maps in order to get pre-strategies on $A  \parallel  B  \parallel 
C$.

Thus we consider
$ \sigma \parallel \text{id}_C : S \parallel C \rightarrow A \parallel
B \parallel C$
and
$\text{id}_A \parallel \tau : A \parallel T \rightarrow A \parallel
B \parallel C$.
Since the identity map on any $A$ accepts all possible behaviour
appearing in $A$, only $ \sigma $ and $ \tau $ give constraints on $A$
and $C$ respectively. In our example, the interaction is:

$$\xymatrix@R=10pt@C=20pt{
\ffalse_2 & & \ttrue_2 \\
\ttrue_1 \ar@{-|>}[u] \ar@{~}[rr] & &  \ffalse_1\ar@{-|>}[u] \\
& ( \sigma   \parallel  \text{id}_C) \wedge (\text{id}_A  \parallel   \tau )
}$$

This interaction
will be written $ \tau   \circledast   \sigma  : T  \circledast  S  \rightarrow  A  \parallel  B  \parallel  C$. 
(Note the change of order from $( \sigma   \parallel  \text{id}_C) \wedge (\text{id}_A  \parallel   \tau )$
to $\tau \circledast \sigma$, which reflects the standard notation for composition. In particular, when $A = C = 1$,
$ \sigma  \wedge  \tau $ is the same as $ \tau   \circledast   \sigma $.)

\paragraph{Hiding.}\label{par:hiding}
 From $ \tau   \circledast   \sigma  : T  \circledast  S  \rightarrow  A  \parallel  B  \parallel  C$ we need to obtain a
map to $A  \parallel  C$. For an event $p\in T  \circledast  S$ we say that it is \textbf{visible} if it maps
to $A$ or $C$, \textbf{invisible} otherwise. Let us write $V$ for the set of visible events of $T  \circledast  S$.

We now obtain the composition by hiding invisible events:
formally, $T  \odot  S = (T  \circledast  S)  \downarrow  V$. The obvious function
$ \tau   \odot   \sigma  : T  \odot  S  \rightarrow  A  \parallel  C$, got as the restriction of $
\tau   \circledast   \sigma${,} defines a map of event structures. Polarities
on $T  \odot  S$ are inherited from those of $A^\perp  \parallel  C$ to make $ \tau   \odot   \sigma $
a pre-strategy on $A^\perp  \parallel  C$.
In our example this yields the pre-strategy on $\mathbb B$ (notice the inheritance of conflict -- the conflict between $\ffalse_2$ and
$\ttrue_2$ becomes minimal after hiding):

$$\xymatrix@R=10pt{\ffalse_2^+ \ar@{~}[rr] & & \ttrue_2^+ \\ &  \tau   \odot   \sigma }$$

We get back the original nondeterministic boolean -- the non-deterministic boolean is
invariant under negation. But in what sense is it the same, exactly?
%

\paragraph{Isomorphisms of pre-strategies.}
They are not equal (set-theoretically) because the underlying sets are not the same, but they are
\emph{isomorphic}:
\begin{defi}[Isomorphism of pre-strategies]
  Let $ \sigma  : S  \rightarrow  A$ and $ \tau  : T  \rightarrow  A$ be two pre-strategies on a common
  game $A$. {An \textbf{isomorphism} between $\sigma$ and $\tau$ is an isomorphism}
  of event structures $ \phi  : S  \cong  T$ commuting with the action on the
  game:

  $$\xymatrix{
    \ar[dr]_{ \sigma }    S \ar@/^/[rr]^{ \phi } & & T \ar[dl]^{ \tau } \\
    & A
  }$$
  In this case, we write $ \phi  :  \sigma   \cong   \tau $ or simply $ \sigma   \cong   \tau $
{when the specific $\phi$ does not matter.}
\end{defi}
Isomorphism is the most precise equivalence that makes
sense on pre-strategies: two isomorphic pre-strategies have the same
intensional behaviour.

Constructing isomorphisms at the level of events can be sometimes cumbersome
especially in the case when the event structures are generated from
an order of configurations as is the case for the interaction (Section
\ref{sec:pb}). Fortunately, order-isomorphisms between configurations
of event structures induce isomorphisms on the event structures ({\it cf.}~\cite{NPW}).
\begin{lem}\label{lemma:eviso}
  Let $A$ and $B$ be event structures. {For every order-isomorphism
  $ \varphi :  \mathscr{C} (A)  \cong   \mathscr{C} (B)$ there is a unique isomorphism of event structures
  $\hat { \varphi } : A  \cong  B$ satisfying $\hat { \varphi }(x) =  \varphi(x)$ for every configuration
  $x  \in   \mathscr{C} (A)$. This induces a bijective correspondence between order-isomorphisms
  $\mathscr{C} (A)  \cong   \mathscr{C} (B)$ and isomorphisms of event structures $A  \cong  B$.}
\end{lem}
\begin{proof}
  Since it is an order-isomorphism, $ \varphi $ preserves the covering relation on configurations.
We define $\hat{\varphi}$ through this property: indeed for all $a\in A$ we have $[a) \cov [a]$,
therefore $\varphi\,[a) \cov \varphi\,[a]$; let us write $\hat{\varphi}\,a$ for the event added in
this covering. In order to establish that $\hat{\varphi}$ is a map of event structures whose image of
configurations matches $\varphi$, the key property will be that for all $x \cov x \cup \{a\}$, the event 
added in $\varphi\,x \cov \varphi(x\cup \{a\})$ is indeed $\hat{\varphi}\,a$.

For that, we remark that $\varphi$ preserves commuting squares of coverings of the form:

  $$\xymatrix {
    y_1 \ar@{}[r]|{\cov}^{a_1} & z \\
\ar@{}[u]|{\rotatebox{90}{$\cov$}}^{a_2}
      x \ar@{}[r]|{\cov}^{a_1} & y_2 \ar@{}[u]|{\rotatebox{90}{$\cov$}} ^{a_2}
 }$$
in the sense that their images are squares where parallel arrows correspond to the same event.
 Since $\varphi$ preserves $\cov$, the image of the square as above is:

  $$\xymatrix {
    \varphi\,y_1 \ar@{}[r]|{\cov}^{b'_1} & \varphi\,z \\
\ar@{}[u]|{\rotatebox{90}{$\cov$}}^{b_2}
      \varphi\,x \ar@{}[r]|{\cov}^{b_1} & \varphi\,y_2 \ar@{}[u]|{\rotatebox{90}{$\cov$}} ^{b'_2}
 }$$

If $\varphi\,y_1 = \varphi\,y_2$ then $y_1 = y_2$ so $a_1 = a_2$ and $a_1 \in y_1$, contradicting
$y_1 \longcov{a_1}$. So since $\varphi\,y_1 \neq \varphi\,y_2$ we have $b_2 \neq b_1$, therefore 
  $b_1 = b'_1$ and $b_2 = b'_2$. 

Now, by induction on $x$ we prove that $\hat{ \varphi } x =  \varphi x$. Clearly
$\varphi \emptyset = \hat{\varphi} \emptyset = \emptyset$.
Now take $x \longcov{a} y$, and write $ \varphi x \longcov{b} \varphi
  y$. We have covering diagrams as below:

  $$\xymatrix{
 x \ar@{}[r]|{\longcov a} & y  & &  \varphi x \ar@{}[r]|{\longcov b} &  \varphi y\\
\ar@{}[u]|{\rotatebox{90}{$\cov$}}\vdots & \ar@{}[u]|{\rotatebox{90}{$\cov$}} \vdots \ar@{}[rr]|{\mapsto} & & \ar@{}[u]|{\rotatebox{90}{$\cov$}}\vdots & \ar@{}[u]|{\rotatebox{90}{$\cov$}}\vdots\\
[a) \ar@{}[r]|{\longcov a}\ar@{}[u]|{\rotatebox{90}{$\cov$}} &[a] \ar@{}[u]|{\rotatebox{90}{$\cov$}} & &  \varphi [a) \ar@{}[r]|{\longcov {\hat{ \varphi }a}}\ar@{}[u]|{\rotatebox{90}{$\cov$}} &  \varphi [a] \ar@{}[u]|{\rotatebox{90}{$\cov$}}
}$$
where the left hand side diagram decomposes into commuting squares of coverings where all horizontal
coverings add $a$. Since those are
preserved, it follows that $b = \hat{\varphi}\,a$. Hence $\hat{\varphi}\,y = \hat{\varphi}\,x \cup
\{\hat{\varphi}\,a\} = \varphi\,x \cup \{b\} = \varphi\,y$.

Obviously it follows that $\hat{\varphi}$ preserves configurations. It is also locally injective
since $x$ and $\varphi\,x = \hat{\varphi}\,x$ have the same cardinal (as $\varphi$ preserves
coverings). Thus $\hat{\varphi}$ is a map of event structures. From Lemma \ref{lem:faithful} it
follows that $\hat{\varphi}$ and $\hat{\varphi^{-1}}$ are inverses.

{Uniqueness is obvious by Lemma \ref{lem:faithful} and the bijective correspondence follows.}
\end{proof}

It will follow from the developments of Section \ref{sec:bicategory} that up to this notion of
isomorphism of pre-strategies, composition is associative:

\begin{prop}
Let $\sigma : S \to A^\perp \parallel B, \tau : T \to B^\perp \parallel C$ and $\rho : U \to
C^\perp \parallel D$ be pre-strategies. Then, there is an isomorphism $\alpha_{\sigma, \tau, \rho} :
(U\odot T) \odot S \to U \odot (T\odot S)$ making the following diagram commute:
\[
\xymatrix{
(U\odot T) \odot S
	\ar[rr]^{\alpha_{\sigma, \tau, \rho}}
	\ar[dr]_{(\rho \odot \tau) \odot \sigma}&&
U \odot (T\odot S)
	\ar[dl]^{\rho \odot (\tau \odot \sigma)}\\
&A^\perp \parallel D
}
\]
\end{prop}
\begin{proof}
The isomorphism is constructed in Section \ref{sec:assoc}.
\end{proof}

{To accompany this associative composition, the next section will start by introducing
a \emph{copycat pre-strategy}, that serves as a candidate for an identity. The copycat pre-strategy
\[
\ccc_A : \CCC_A \to A^\perp \parallel A
\]
is an \emph{asynchronous forwarder}: every negative move on one side triggers the corresponding
positive move on the other side.
It is idempotent, but we will see that it is not an identity:
pre-strategies and copycat \emph{do not form a category}. We will define
\emph{strategies} as those pre-strategies for which copycat is an identity, and characterise them
concretely.}

\section{Strategies
}
\label{sec:strategies}
As previously hinted at, pre-strategies currently 
take little account of polarity, and
hence have an unreasonable expressive power: they can for instance constrain
the order in which Opponent plays their 
moves, or prevent 
them
from 
playing at all. {We have encountered just before Section \ref{sec:pb}
a simple example of that: the empty pre-strategy on the game $\ominus$ with just
one negative move. By not acknowledging the $\ominus$, Player denies Opponent the right
to play, even though the game allows it.}

One guiding principle for the notion of strategy is that {they should form a category}, so there
should be a copycat strategy, {neutral for composition with respect to other strategies}. {The
presence of an identity for composition} is of
course key to the application of our setting in denotational semantics, 
which relies on a categorical 
formalisation, but we argue that there is a more
down-to-earth motivation for {it. The \emph{copycat strategy}, to be introduced
formally below, acts as an \emph{asynchronous forwarder}. Accordingly, composing with copycat will eliminate overly \emph{synchronous} behaviour from pre-strategies.  
Examples
include the pre-strategy in the previous paragraph not acknowledging Opponent's move, or a pre-strategy
playing on a game $\oplus_1 \oplus_2$ with two independent positive events, which plays the moves
$\oplus_1 \imc \oplus_2$ in order.
As the moves are independent in the game, the ordering $\oplus_1 \imc \oplus_2$ played by the strategy
will not be respected by an asynchronous environment -- two successive packets sent on the network
might arrive in the other order. These intuitions will be revisited formally via examples after the
definition of the copycat pre-strategy.}

In this section we will define the copycat (pre\-)strategy,
and then characterise the \emph{strategies}: those pre-strategies invariant
under their composition with copycat.
We provide examples of pre-strategies that do not behave
well in presence of latency and give two criteria (\emph{courtesy} and
\emph{receptivity}) that are proved necessary and sufficient for a
pre-strategy to be a strategy (Theorem \ref{thm:main_thm}).

\subsection{Copycat and its action on strategies}
\newcommand{\Click}{Click} 
\newcommand{\Compute}{Done}

%
On $A^\perp \parallel A$, each move of $A$ appears twice (with dual
polarities). The copycat pre-strategy waits for a negative occurrence to be played
and then plays the corresponding positive move. In formal terms, it
has the causality $(1-i, a) \imc (i, a)$ for every positive move $(i, a)$
of $A^\perp \parallel A$. Note that this behaviour corresponds to that of the
usual copycat strategy in game semantics.

For instance, on the game
$\mathbb W = \mathbf{\Click}^-\ \ \mathbf{\Compute}^+$ of an interface
where Player (the program) can signal it has finished a long
computation or Opponent (the user) can click on the screen, the
copycat strategy looks like:
$$
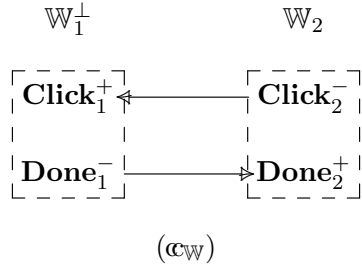
\xymatrix@C=0.3cm@R=10pt {
\mathbb W^\perp_1 & & \mathbb W_2 \\
\mathbf{\Click}_1^+ & & \mathbf{\Click}_2^- \ar@{-|>}[ll] \\
\ar@{-|>}[rr]\mathbf{\Compute}_1^- & & \mathbf{\Compute}_2^+ \\
& (\ccc_{\mathbb W})
\save "2,1". "3,1" *[F--]\frm{} \restore
\save "2,3". "3,3" *[F--]\frm{} \restore
}$$

Copycat forwards the negative events from one side to the other:
acting as the program on the right and as the user on the left. Even
if copycat is a pre-strategy from $\mathbb W$ to itself, it does not necessarily 
entail a left-to-right flow of information as can be seen for the
event \textbf{Click}, rather \emph{from negative to positive}. This
general construction yields a pre-strategy playing on
$A^\perp \parallel A$ for any game $A$.

\begin{defi}[Copycat]
  Let $A$ be a game. Define $\CCC_A$ to be the following event structure:
  \begin{itemize}
  \item \emph{Events:} those of $A^\perp  \parallel  A$,
  \item \emph{Causality}: the transitive closure of
$$ \leq _{A^\perp  \parallel  A} \cup \{ \left((1-i, a), (i, a)\right) \mid (i, a)^+  \in  A^\perp  \parallel  A \}$$
\item \emph{Consistency}: $X$ is consistent in $\CCC_A$ iff its
  down-closure $[X] = \{ a  \in  \CCC_A \mid  \exists b  \in  X, a  \leq _{\CCC_A} b\}$ is
  consistent in $A^\perp  \parallel  A$.
  \end{itemize}
\end{defi}

\noindent This makes an event structure and the identity map is a pre-strategy:

\begin{lem}
For any game $A$, $\CCC_A$ is {an esp (with polarities inherited from $A^\perp \parallel A$), and the
identity map written $\ccc_A : \CCC_A  \rightarrow  A^\perp  \parallel  A$ is a pre-strategy, the
\textbf{copycat pre-strategy}.}
\end{lem}
\begin{proof}
We observe that for $(i, a), (j, a') \in \CCC_A$, we have $(i,a) \leq_{\CCC_A} (j,a')$ iff:
\begin{itemize}
\item Either, $i = j$ and $a\leq_A a'$,
\item Or, $i\neq j$, and there is $a\leq_A a'' \leq_A a'$ such that 
$\pol_{\CCC_A}((i,a'')) = - $
and (by necessity) $\pol_{\CCC_A}((j,a'')) = +$. 
\end{itemize}

\noindent Indeed, this is a transitive relation that contains the generators for $\leq_{\CCC_A}$ -- dually, two
events related by the relation above are related by $\leq_{\CCC_A}$. The other axioms of event
structures follow easily, and it is trivial that $\ccc_A : \CCC_A \to A^\perp \parallel A$ is a map of
event structures.
\end{proof}

Immediate causal links in copycat have a very specific shape:
\begin{lem}
  \label{lemma:imc_cc} We have that $(i, a)  \rightarrowtriangle _{\CCC_A} (j, a')$ if and
  only if one of the two following conditions is met:
  \begin{enumerate}
  \item Either $i = j$, $a \rightarrowtriangle _A a'$ and either
    $(i, a)$ is positive in $\CCC_A$ or $(j, a')$ is negative in
    $\CCC_A$.
  \item Or $i \neq j$ and $a = a'$ and $(i, a)  \in  \CCC_A$ is negative.
  \end{enumerate}
\end{lem}
\begin{proof}
  It is clear that both conditions imply
  $(i, a) \rightarrowtriangle _{\CCC_A} (j, a')$.  Conversely, we know
  $ \leq _{\CCC_A}$ is generated by
  $ \rightarrowtriangle _{A^\perp \parallel A} \cup \{((i, a), (1-i,
  a) \mid (i, a)^- \in \CCC_A\}$.
  This means that $(i, a) \rightarrowtriangle (j, a')$ implies either
  $i \neq j$, $a = a'$ and $(i, a)^- \in \CCC_A$ (as desired) or
  $i = j$ and $a \rightarrowtriangle _A a'$. In this case, if $(i, a)$
  is negative and $(j, a')$ is positive, we have
  $(i, a) \rightarrowtriangle _{\CCC_A} (1-i, a) <_{\CCC_A} (1-i, a')
  \rightarrowtriangle _{\CCC_A} (i, a')$
  contradicting $(i, a) \rightarrowtriangle _{\CCC_A} (j, a')$. Hence
  $(i, a)$ is positive.
\end{proof}

Copycat acts on pre-strategies on $A$ via composition:
$ \sigma \mapsto \ccc_A \odot \sigma $. This action adds \emph{latency}
to pre-strategies: whenever the pre-strategy plays a positive move it
has to be forwarded by copycat before being visible. We can now define
strategies:

\begin{defi}[Strategy]\label{def:strategy}
  A \textbf{strategy} on a game $A$ is a pre-strategy $ \sigma  : S  \rightarrow  A$ such
  that $\ccc_A  \odot   \sigma   \cong   \sigma $.
\end{defi}

{This isomorphism is, in general, not unique: in fact, strategies have
in general a non-trivial group of automorphisms. Think, for instance, of the strategy $\sigma : S
\to \mathbb B$ where $S$ has just two conflicting events $\xymatrix@C=10pt{s \ar@{~}[r]&s'}$, both
mapped to $\ttrue$. This strategy $\sigma$ has two automorphisms: the identity and the swap on $S$. 
Likewise, there are two isomorphisms $\ccc_{\mathbb B} \odot \sigma \cong \sigma$.
Despite this, it will follow from our development that if $\ccc_A  \odot   \sigma   \cong   \sigma $,
then there is always
a \emph{canonical} such isomorphism that fits in the bicategorical picture of Section
\ref{sec:bicategory}.}

Let us try to understand this definition through examples.
Consider first the composition $\ccc_{\mathbb W} \odot \ccc_{\mathbb W}$ with
$A = \mathbb W_1, B = \mathbb W_2$ and $C = \mathbb W_3$:

$$
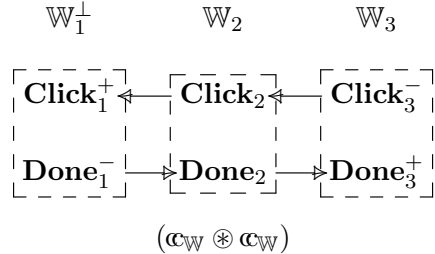
\xymatrix@C=0.4cm@R=10pt{
\mathbb W_1^\perp & \mathbb W_2 & \mathbb W_3 \\
\mathbf{\Click}^+_1 &\ar@{-|>}[l] \mathbf{\Click}_2 & \mathbf{\Click}^-_3 \ar@{-|>}[l]\\
\ar@{-|>}[r]\mathbf{\Compute}^-_1 & \ar@{-|>}[r] \mathbf{\Compute}_2 & \mathbf{\Compute}^+_3 & & & \\
& *\txt{$(\ccc_{\mathbb W}  \circledast  \ccc_{\mathbb W})$}
\save "2,1"."3,1" *[F--]\frm{} \restore
\save "2,2"."3,2" *[F--]\frm{} \restore
\save "2,3"."3,3" *[F--]\frm{} \restore
}$$\medskip

Hiding events in $\mathbb W_2$ yields a pre-strategy isomorphic to
$\ccc_{\mathbb W}$. The latency can be observed: immediate
causal links of the form $-  \rightarrowtriangle  +$ get delayed in the interaction to
$-  \rightarrowtriangle   {_\ast}   \rightarrowtriangle  +$ where $ {_\ast} $ denotes an invisible event of the
interaction. After hiding, the effect disappears here but it is not the
case in general. Two situations can appear, calling for two conditions.

\paragraph{Courtesy.} Assume we have the pre-strategy $\sigma$
with event structure $\mathbf{\Compute}^+ \rightarrowtriangle \mathbf{Click}^-$
on $\mathbb W$ that forces the user to wait for the computation to be
over before allowing them to click. Computing the interaction
$\ccc_{\mathbb W} \circledast \sigma $ with $A = {1}$, $B = \mathbb W_1$ and
$C = \mathbb W_2$ yields:
$$
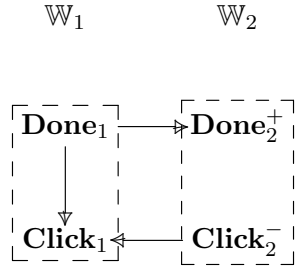
\xymatrix{
\mathbb W_1 & \mathbb W_2 \\
\mathbf{\Compute}_1 \ar@{-|>}[d]\ar@{-|>}[r] & \mathbf{\Compute}_2^+ \\
\mathbf{\Click}_1 & \mathbf{Click}_2^- \ar@{-|>}[l]
\save "2,1"."3,1" *[F--]\frm{} \restore
\save "2,2"."3,2" *[F--]\frm{} \restore
}$$

After hiding of $B = \mathbb W_1$,
$\ccc_{\mathbb W} \odot \sigma$ has event structure $\mathbf{\Click}_2^-\ \
\mathbf{\Compute}_2^+$.
There is no causal link anymore because in the interaction the two
events are concurrent. Copycat \emph{will} allow the user to Click
without waiting for $ \sigma $'s constraint: there is no way for
$ \sigma $ to impose this particular order of moves. In other terms the causal link is not stable under the latency added by copycat.

As a consequence, for a pre-strategy to be invariant under the action
of copycat it must not have immediate causal links of the form
$+ \rightarrowtriangle - $ \emph{that were not already present in the
  game}. In our setting, playing a move is similar to sending a packet
whose sender (Player or Opponent) is given by the polarity. This
condition means that unless the protocol (the game) specifies it,
there is no way to force Opponent to wait for a Player message before
sending their message.

Similar reasoning can be made for immediate causal links $-   \rightarrowtriangle - $ (one cannot
control the order in which Opponent sends out messages) and $+   \rightarrowtriangle +$
(latency can change the order in which independent messages arrive).

A pre-strategy respecting these constraints will be called
\emph{courteous}\footnote{This condition was called \emph{innocence} in \cite{lics11}. Courtesy is
preferred here to avoid the misleading collision with innocence in the sense of Hyland and Ong
\cite{DBLP:journals/iandc/HylandO00}.}:

\begin{defi}[Courtesy]
  A pre-strategy $ \sigma  : S  \rightarrow  A$ is \textbf{courteous} when for all $s, s'  \in  S$
  such that $s  \rightarrowtriangle  s'$ and $(\pol(s), \pol(s')) \neq (-, +)$, then
  $ \sigma s  \rightarrowtriangle   \sigma s'$.
\end{defi}

\paragraph{Receptivity.} Consider the game $\mathbb Y = \mathbf{o}^-$
comprising a single negative event, and the two pre-strategies $\sigma$
and $\tau$ on this game, with respective event structures $\emptyset$ (no moves played by $\sigma$)
and $\xymatrix@C=0.3cm { \mathbf{o}^- \ar@{~}[r] & \mathbf{o}^- }$ ($\tau$ can acknowledge the 
unique negative event in two different ways, non-deterministically).

Their respective interactions with copycat on $\mathbb Y$ give (with
$A = {1}, B = \mathbb Y_1$ and $C = \mathbb Y_2$):

$$
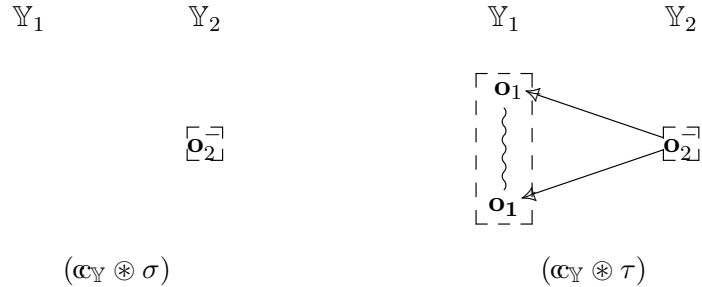
\xymatrix@R=10pt@C=0pt{
\mathbb Y_1 &&\mathbb Y_2 & \hspace{80pt}& & \mathbb Y_1 && \mathbb Y_2 \\
&&& & & \txt{ $\textbf{o}_1$ }\ar@{~}[dd] \\
&&*[F--]\txt{$\mathbf{o}^-_2 $}& & & && *[F--]\txt{$\mathbf{o}_2^-$}\ar@{-|>}[lld]\ar@{-|>}[llu] \\
&&& & & \mathbf{o_1} \\
&(\ccc_{\mathbb Y}  \circledast   \sigma ) &&& & & (\ccc_{\mathbb Y}  \circledast   \tau )
\save "2,6". "3,6". "4,6" *[F--]\frm{} \restore
}$$

After hiding, only $\mathbf{o}^-_2$ is left in both cases. The problem
with these pre-strategies is that they either duplicate or ignore a
negative event -- yet as we have seen, copycat acknowledges available negative moves
\emph{first} without depending on the pre-strategy's
behaviour. Strategies must therefore have the same behaviour regarding
the negative events as copycat: to accept them as soon as they are
enabled in the current state of the game, and play them \emph{once}. Such
pre-strategies will be called \emph{receptive}:

\begin{defi}[Receptivity]
  A pre-strategy $ \sigma  : S  \rightarrow  A$ is \textbf{receptive} when for each
  configuration $x  \in   \mathscr{C} (S)$ such that $ \sigma x \longcov{a^-}$ there exists a \emph{unique} $s  \in  S$ (necessarily negative) such that $x \longcov s$ and $ \sigma s = a$.
\end{defi}

For readers familiar with game semantics, it might be helpful to note that in standard games models receptivity is always present in one way or another.
It is explicit and named \emph{contingent completeness} in \cite{DBLP:journals/iandc/HylandO00}, but most of the time
it is hard-wired in by asking that strategies contain only plays of even length (Opponent extensions being always present, they bring no additional information).

\subsection{The characterisation of strategies -- overview of the proof}

At this point, the main definitions for the framework are in place. The main element which is
missing, is the fact that for a pre-strategy $\sigma : S \to A$, it is equivalent to be a
\emph{strategy} (in the sense of Definition \ref{def:strategy}), and to be \emph{receptive} and
\emph{courteous} -- which was the main result of \cite{lics11}. The rest of this section is devoted
to proving this result, stated in Theorem \ref{thm:main_thm}. 
In this paper we give a different proof than the one developed in \cite{lics11}. Our new proof is
more high-level and modular, and sets up the stage better for extensions of the framework in future
papers. 
The rest of the section is quite technical, and may be skimmed through in a first reading of the
paper. We start by giving a high-level overview of the proof. 

According to Definition \ref{def:strategy}, $\sigma : S \to A$ is a strategy if $\ccc_A \odot \sigma \cong \sigma$.
By Lemma \ref{lemma:eviso}, that means that there is a order-isomorphism
\[
\conf{S} \cong \conf{\CCC_A \odot S}
\]
commuting with the projection to $A$. In order to characterise the existence of such an isomorphism,
we need to study configurations of $\CCC_A \odot S$ for any pre-strategy $\sigma : S \to A$. This will be done is several steps.

\paragraph{Decomposing interactions.} 
Taking $z \in \conf{\CCC_A \odot S}$, we have its minimal witness $[z] \in \conf{\CCC_A \circledast S}$. By Lemma \ref{lemma:confinter},
$[z]$ corresponds to a secured bijection:
\[
\varphi_{[z]} : x \simeq y
\]
with $x = x_S \parallel x_A \in \conf{S\parallel A}$ and $y = y_{A^\perp} \parallel y_A \in \conf{\CCC_A}$ such that
$\sigma\,x_S = y_{A^\perp}$ and $x_A = y_A$ -- in fact, as remarked below Lemma \ref{lemma:confinter}, by local
injectivity, $\varphi_{[z]}$ (and so $[z]$) is determined by such $x$ and $y$, \emph{i.e.}, by $x_S$ and $y_A$.

We write $\Psi([z]) = (x_S, y_A) \in \conf{S} \times \conf{A}$ for this pair, which satisfies that $x_S \in \conf{S}$ and
$\sigma x_S \parallel y_A \in \conf{\CCC_A}$. Reciprocally (by Lemma \ref{lemma:confinter}) any such pair induces
a configuration of $\CCC_A \circledast S$ provided the corresponding bijection is secured -- but that is always the case,
as we will see; so $\Psi$ is an iso. We will also characterise 
such pairs which, through $\Psi$, correspond
to an interaction whose maximal elements are visible (\emph{i.e.} a minimal witness of a configuration of $\CCC_A \odot S$).
This will yield a complete description of configurations of $\CCC_A \odot S$ in terms of certain pairs of configurations
$(x, y) \in \conf{S} \times \conf{A}$ (step \#1).

For $z \in \conf{\CCC_A \circledast S}$ and $\Psi(z) = (x_S, x_A)$, one may regard $x_A$ as a not completely updated 
version of $\sigma\,x_S$: some negative events of $x_A$ may not have made their way to $\sigma\,x_S$, and reciprocally.

\begin{exa}\label{ex:inter}
Consider the pre-strategy $\sigma$ playing on $\mathbb W_1 \parallel \mathbb W_2$, with event structure $\mathbf \Compute^+_1 \imc \mathbf \Compute_2^+$.
The following diagram represents an interaction $z\in \conf{\CCC_A\circledast S}$ of $\sigma$ with copycat.

\[
\xymatrix@R=10pt{
(\mathbb W_1	\ar@{}[r]|\parallel&
\mathbb W_2)&
(\mathbb W_1	\ar@{}[r]|\parallel&
\mathbb W_2)\\
\mathbf \Compute_1
	\ar@{-|>}[dr]& \text{\phantom{\Compute}} & \mathbf \Click_1^-&
\mathbf \Compute_2^+\\
\text{\phantom{$\Compute_2$}}&\mathbf \Compute_2
	\ar@{-|>}[rru]  & \text{\phantom{$\Compute_2$}} & \text{\phantom{$\Compute_2$}} &
\save "2,1" . "3,1" *[F--]\frm{} \restore
\save "2,2" . "3,2" *[F--]\frm{} \restore
\save "2,3" . "3,3" *[F--]\frm{} \restore
\save "2,4" . "3,4" *[F--]\frm{} \restore
}
\]\medskip

\noindent Here, we have $\Psi(z) = (\{\mathbf \Compute_1^+, \mathbf \Compute_2^+\}, \{\mathbf \Click_1^-, \mathbf \Compute_2^+\})$. 
\end{exa}

In the example above, we observe two phenomena: the event $\mathbf \Click_1^-$ is played on the right hand side but not forwarded to
the left hand side, and the event $\mathbf \Compute_1^+$ is played on the left hand side but not forwarded to 
the right hand side. In general, with $\Psi(z) = (x_S, x_A)$, the constraint that $\sigma\,x_S \parallel x_A \in \conf{\CCC_A}$ means 
that $\sigma\,x_S$ has \emph{less negative events} and \emph{more positive events} than $x_A$,
\emph{i.e.}
\[
x_A \supseteq^- x_A \cap (\sigma\,x_S) \subseteq^+ \sigma\,x_S
\]
This relation $x \supseteq^- \subseteq^+ y$ is in fact a partial order on $\conf{A}$ called the
\emph{Scott order} \cite{DBLP:conf/fossacs/Winskel13}, {and written $\sqsubseteq_A$. It} will yield (step \#2)
a characterisation of configurations of copycat as pairs $(x_S, x_A) \in \conf{S}\times \conf{A}$
such that $x_A \sqsubseteq_A \sigma\,x_S$.

To summarise, after steps \#1 and \#2, we will have achieved an equivalent description of
interactions $z \in \conf{\CCC_A \circledast S}$ as the data of $(x_S, x_A) \in \conf{S}\times
\conf{A}$ such that $x_A \sqsubseteq_A \sigma\,x_S$, \emph{i.e.} as diagrams:
\[
\xymatrix{
&x_S    \ar@{|->}[d]^\sigma\\
x_A     \ar@{}[r]|{\sqsubseteq_A}& 
\sigma\,x_S
}
\]
whose projection to the game via $\ccc_A \circledast \sigma : \CCC_A \circledast S \to A\parallel A$ is
$\sigma\,x_S \parallel x_A$, where only $x_A$ will be visible after hiding.
We now try to produce an isomorphism between configurations of $\CCC_A \circledast S$ that are minimal witnesses of
configurations of $\CCC_A \odot S$ (those whose maximal events are visible), and configurations of $S$. We will build
transformations of configurations in the two directions.

\paragraph{The isomorphism.} 
Constructing the left-to-right part of the isomorphism $\CCC_A\odot S \cong S$, we need to associate
to any representation of an interaction $(x_S, x_A) \in \conf{S} \times \conf{A}$ as above, 
some $x'_S \in \conf{S}$ mapping to $x_A$ via $\sigma$. Diagrammatically:
\[
\xymatrix{
&x_S	\ar@{|->}[d]^\sigma\\
x_A	\ar@{}[r]|{\sqsubseteq_A}& 
\sigma\,x_S
}
~~~~~~~~~~\implies~~~~~~~~~~
\xymatrix{
\exists x'_S 	\ar@{|->}[d]^\sigma&x_S    \ar[d]^\sigma\\
x_A     \ar@{}[r]|{\sqsubseteq_A}& 
\sigma\,x_S
}
\]
In fact, it will turn out that $x'_S \sqsubseteq_S x_S$, and (for the correspondence to be an iso) that its
choice is unique. In other words, we will extract $x'_S$ by proving that strategies are
\emph{discrete fibrations}, as in Definition \ref{def:disfib} (step \#3).

We now focus on the right-to-left part of the construction. From $x\in \conf{S}$, we need to provide some configuration
of $\CCC_A \odot S$; so we need to provide a witness in $\conf{\CCC_A \circledast S}$. As we have seen,
via $\Psi$ we are looking for a pair $(x_S, x_A)$ such that $x_A \sqsubseteq_A \sigma\,x_S$. 
Note that $x_A$ is 
determined 
by the requirement that $\sigma\,x = x_A$. From that it seems that the pair $(x, x_A)$ does
the trick: we do indeed have $z = \Psi^{-1}(x, x_A) \in \conf{\CCC_A \circledast S}$ -- and restricting it to its
\emph{visible} events yields the desired configuration of $\CCC_A \odot S$. However, it will be useful in proving the isomorphism
to have the \emph{minimal} interaction -- the minimal witness -- corresponding to this configuration of the composition through
hiding. The interaction $\Psi^{-1}(x, x_A)$ is not always minimal:

\begin{exa}
Consider $\sigma : S \to \mathbb W$ with $S$ comprising only one event $s$ mapped to $\mathbf \Click_1^-$. Following the
paragraph above, its configuration $\{s\}$ leads to an interaction with copycat corresponding to
$(\{s\}, \{\mathbf \Click_1^-\})$, represented as:
\[
\xymatrix@R=0pt{
\mathbb W
	\ar@{}[r]|\parallel&
\mathbb W\\
\mathbf \Click_1&
\mathbf \Click_1^-
	\ar@{-|>}[l]
}
\]
Disposing of the left hand side $\mathbf \Click_1$ yields a \emph{smaller} interaction witnessing the same
configuration of the composition, as it is maximal and not visible.
\end{exa}

In fact, for $x\in \conf{S}$ there is a \emph{unique} $x^* \subseteq x$ such that $(x^*, \sigma\,x)$ yields
the same configuration of the composition as $(x, \sigma\,x)$, and such that the maximal events of the
represented interaction are all visible. As we will see $x^*$ is obtained from $x$ as above,
by removing maximal negative events (step \#4). From this uniqueness property and the discrete fibration
property, it follows that these constructions are inverses of each other.

\paragraph{Necessity.} From the above, we know that strategies, as discrete fibrations, compose well with
copycat. It remains to show the converse: that strategies which compose well with copycat are discrete fibrations.
In other words, we need to show that strategies of the form $\ccc_A \odot \sigma$ are always discrete fibrations.
That will be a direct verification, once we have characterised the Scott order on $\CCC_A \odot S$ (step \#5).

\subsection{Proof of the characterisation of strategies}

Now, we detail and prove all the steps mentioned above.

\subsubsection{Step \#1: Composition witnesses as pairs}

We start by showing that there are no possible causal loops in an interaction with copycat, so that such 
interactions are entirely characterised by matching pairs of configurations. In fact we prove a slight generalisation. 

\begin{lem}[Deadlock-free lemma]\label{lemma:deadlock}
  Let $\tau : T \rightarrow A^\perp \parallel B$ be a pre-strategy such
  that if $t \leq t'$ and both $t$ and $t'$ are sent by $ \tau $ to the component
  $A^\perp$, then $ \tau t \leq \tau t'$. Then, given a pre-strategy
  $ \sigma : S \rightarrow A$, and configurations $x$ of $S$ and $y$ of
  $T$ with 
$ \sigma x \parallel z= \tau y$ for some configuration $z$ of $B$
, the induced bijection
$x \parallel z \simeq y$
 is secured.

  As a consequence, we have an order isomorphism:
  $$ \mathscr{C} ( T   \circledast S )  \cong  \{ (x, y)  \in   \mathscr{C} (S)  \times   \mathscr{C} (T) \mid  
\sigma x \parallel z =  \tau y
\hbox{ for some } z \hbox{ in }  \mathscr{C}(B)
 \}$$
\end{lem}
\proof
  Assume that the bijection is not secured. Without loss of generality,
  there is a causal loop of the form
  $(v_1, t_1) \vartriangleleft \ldots \vartriangleleft (v_{2n},
  t_{2n})$
  such that $t_{2i} < t_{2i+1}$ and $v_{2i+1} < v_{2i+2}$
  and $t_{2n} < t_1$. Note that $v_i  \in  S  \parallel  B$ for every $i$.

  Assume that $v_{2i+1} \in B$. Then $v_{2i+2} \in B$ and we have that
  $\tau(t_{2i+1}) = v_{2i+1}  \leq  v_{2i+2} = \tau(t_{2i+2})$.  Hence by
  Lemma \ref{lemma:reflect}, it follows that $t_{2i+1}  \leq 
  t_{2i+2}$.  If the only
  two steps of the causal loop were $(v_{2i+1}, t_{2i+1})$ and $(v_{2i}, t_{2i})$, we
  have a loop in $T$ and a contradiction. Otherwise, we can remove
  the steps $2i+1$ and $2i+2$ and keep a causal loop. Removing them,
  if there is a loop of length one remaining, then we have a direct contradiction
  ({\it i.e.}~$t_1 < t_1$). Otherwise without loss of generality we can
  assume $v_i \in S$ for every $i$. In this case, by hypothesis on
  $ \tau $ we have that $t_{2i} < t_{2i+1}$ implies that
  $\sigma v_{2i} = \tau t_{2i} < \tau t_{2i+1} = \sigma v_{2i+1}$.
  By Lemma \ref{lemma:reflect} again, it follows
  that $v_1 < \ldots < v_1$ -- a contradiction.

  This establishes that the bijection induced by any pair of
  synchronized configurations $(w, y)$ is secured and thus is a
  configuration of the interaction. We conclude with the sequence of order-isos:
  \begin{align*}
 \mathscr{C} (T \circledast S) & \cong  \{  \varphi :  w \simeq  y \text{ secured} \mid \\
 & \qquad \qquad w  \in   \mathscr{C} (S\parallel B), y  \in   \mathscr{C} (T)\text{ such that }\tau y = ( \sigma   \parallel  B)\,w \} \\
& \cong  \{  \varphi : w  \simeq  y  \mid w  \in   \mathscr{C} (S\parallel B), y  \in   \mathscr{C} (T)\text{ such that } \tau y = ( \sigma   \parallel  B)\,w \} \\ 
& \cong  \{ (x \parallel z, y)  \in   \mathscr{C} (S\parallel B)  \times   \mathscr{C} (T) \mid  \sigma x \parallel z =  \tau y \} \\
& \cong  \{ (x, y)  \in   \mathscr{C} (S)  \times   \mathscr{C} (T) \mid  
\sigma x \parallel z =  \tau y
\hbox{ for some } z \hbox{ in }  \mathscr{C}(B)\rlap{\hbox to 69 pt{\hfill\qEd}}
\}
  \end{align*}\medskip

\noindent Let $ \sigma : S \rightarrow A$ be pre-strategy. 
The previous lemma, instantiated with $ \tau  = \ccc_A$, gives an order-isomorphism:
$$
\begin{aligned}
 \Psi _ \sigma  :  \mathscr{C} (\CCC_A  \circledast   S ) & \cong  \{ (x, y_1  \parallel  y_2)  \in   \mathscr{C} (S)  \times   \mathscr{C} (\CCC_A) \mid  \sigma x = y_1 \} \\
& \cong \{ (x, y)  \in   \mathscr{C} (S)  \times   \mathscr{C} (A) \mid  \sigma x  \parallel  y  \in   \mathscr{C} (\CCC_A) \}
\end{aligned}$$

Every such pair represents an interaction, which gives through hiding a configuration of $\CCC_A \odot S$. However, many
interactions correspond to the same configuration of the composition. In fact, as we
have seen in Section \ref{sec:comp}, configurations of $\CCC_A \odot S$ bijectively correspond
to interactions in $\CCC_A \circledast S$ whose maximal events are visible. We now characterise them.

\begin{lem}\label{lemma:negative_max}
  Let $ \varphi : x \parallel  y  \simeq   \sigma x  \parallel  y$ be a secured bijection corresponding to a
  configuration of $\CCC_A  \circledast  S$. The following are equivalent:
  \begin{enumerate}[label=(\roman*)]
  \item [(i)] All maximal events of $ \varphi $ are visible
  \item [(ii)] Every maximal event $s$ of $x$ is positive and $ \sigma s  \in  y$.
  \end{enumerate}
\noindent Moreover, in this case, if $ \sigma $ is courteous, we have $ \sigma x \subseteq^- y$.
\end{lem}
\begin{proof}\hfill
  \noindent(i) $ \Rightarrow $ (ii).\  Let $s \in x$ be a maximal
    event. The event $c = ((0, s), (0, \sigma s))$ is not visible in
    $ \varphi $.  Hence it is not maximal: there exists
    $c' \in \varphi $ such that
    $c \rightarrowtriangle _{ \varphi } c'$. By Lemma
    \ref{lemma:imc_bij}, there are two cases:
    \begin{itemize}
    \item Either $ \pi _1c  \rightarrowtriangle _{S  \parallel  A}  \pi _1c'$, {\it i.e.}~$c' = ((0, s'), (0,  \sigma s'))$ and $s  \rightarrowtriangle _{x} s'$: this is absurd as
      $s$ is maximal in $x$.
    \item Or $ \pi _2c \rightarrowtriangle _{\CCC_A} \pi _2 c'$: by
      Lemma \ref{lemma:imc_cc}, there are two possibilities. The first one is that
      $c' = ((0, s'), (0, \sigma s'))$: absurd, as it would entail
      $ \sigma s \rightarrowtriangle \sigma s'$ and $s < s'$ by
      Lemma \ref{lemma:reflect} contradicting maximality. The second one is that
      $c' = ((1, \sigma s), (1, \sigma s))$.

      This means that $(1,  \sigma s)$ is positive in $\CCC_A$, {\it i.e.}~$s$ is positive,
      and moreover $(1,  \sigma s)  \in   \sigma x  \parallel  y$ so $ \sigma s  \in  y$.
    \end{itemize}
 
  \noindent(ii) $ \Rightarrow $ (i).\ Let $c$ be a maximal event of $ \varphi $ and assume
    it is not visible. It is then of the form
    $c = ((0, s), (0,  \sigma s))$. If $s  \rightarrowtriangle _x s'$ then
    $c  <_{ \varphi } ((0, s'), (0,  \sigma s'))$ which is absurd so $s$ must be
    maximal in $x$. By assumption $s$ is positive and $ \sigma s  \in  y$. Then
    we have $(0,  \sigma s)  \rightarrowtriangle _{\CCC_A} (1,  \sigma s)$ so
    $c  <_{ \varphi } ((1,  \sigma s), (1,  \sigma s))$ which contradicts the maximality of
    $c$.
  
  Finally, assume $ \sigma $ is courteous. We prove that maximal events of
  $ \sigma x$ are included in $y$. Take $ \sigma s \in \sigma x$ a
  maximal event. If $s$ is negative then $(0,\sigma\,s)$ is positive in $A^\perp \parallel A$.
  Therefore we have $(1, \sigma s) \leq_{\CCC_A} (0, \sigma s)$.
  Since $\sigma x \parallel y \in \conf{\CCC_A}$, 
  we are done. Otherwise, if $s$ is positive it has to be maximal in
  $x$: indeed if we had $s^+ \rightarrowtriangle _x s'$, by courtesy
  $ \sigma s \rightarrowtriangle _{ \sigma x} \sigma s'$ would
  follow contradicting the maximality of $ \sigma s$. Then we can
  conclude by assumption: $ \sigma s \in y$ as desired.
\end{proof}

Summarizing step \#1, we now know that configurations of $\CCC_A \odot S$ correspond, 
in an order-preserving and order-reflecting way, to pairs of configurations $(x, y) \in \conf{S}\times \conf{A}$,
such that $\sigma\,x \parallel y \in \conf{\CCC_A}$, and such that the maximal events of $x$ are positive and also
appear in $y$.

Now, we study the requirement that $\sigma\,x \parallel y \in \conf{\CCC_A}$.

\subsubsection{Step \#2: The Scott order}

As observed before, for $x, y \in \conf{A}$, $y \parallel x \in \conf{\CCC_A}$ whenever $y$ has more positive events
and less negative events than $x$. More precisely:

\begin{lem}[Scott order]\label{lemma:scott}
  Let $x, y  \in   \mathscr{C} (A)$. The following are equivalent:
  \begin{enumerate}[label=(\roman*)]
  \item [(i)] $y  \parallel  x  \in   \mathscr{C} (\CCC_A)$
  \item [(ii)] $x \supseteq^- (x \cap y) \subseteq^+ y$ (where
    $x \subseteq^+ y$ means that $x \subseteq y$ and
    $\text{pol}(y\setminus x) \subseteq \{+\}$ and similarly for $x \supseteq^- y$)
  \item [(iii)] there exists $z  \in   \mathscr{C} (A)$ such that
    $x \supseteq^- z \subseteq^+ y$.
  \end{enumerate}
  In this case we write $x \sqsubseteq_A y$: this is an order
  called \textbf{the Scott order} of $A$.
\end{lem}
\begin{proof}

\noindent {(i) $ \Rightarrow $ (ii).}\ We show $x \cap y \subseteq^+ y$; the other
    inclusion is similar. Let $a^-  \in  y$, we must show it is in
    $x$. Since $(0, a)  \in  A^\perp  \parallel  A$ is positive, we have
    $(1, a) <_{\CCC_A} (0, a)$. The down-closure of $y  \parallel  x$ implies that
    $(1, a)  \in  y  \parallel  x$ as $a  \in  y$. This exactly means that $a  \in  x$ as
    desired.\medskip

\noindent{(ii) $ \Rightarrow $ (iii).}\ clear.\medskip

\noindent {(iii) $ \Rightarrow $ (i).}\ Assume we have $x \supseteq^- z \subseteq^+ y$.
    The set $y  \parallel  x$ is clearly consistent so we need only prove it is
    down-closed. Since $x$ and $y$ are already down-closed in $A$, we
    need only to check for the additional immediate causal
    links. Assume we have $(1, a^+)  \in  y  \parallel  x$ (so $a  \in  x$). By
    hypothesis we have $a  \in  z$ because it is positive. Since
    $z \subseteq y$ we deduce $a  \in  y$ that is $(0, a)  \in  y  \parallel  x$ as
    desired. The case $(0, a^-)  \in  y  \parallel  x$ is similar.

\emph{It is an order.} It is clearly reflexive. If
    $x \supseteq^- (x \cap y) \subseteq^+ y$ and
    $y \supseteq^- (x \cap y) \subseteq^+ x$, it follows that
    $x \setminus x \cap y$ has to be empty thus $x = x \cap y = y$.

    For transitivity assume
    $x \supseteq^- (x \cap y) \subseteq^+ y \supseteq^- (y \cap z)
    \subseteq^+ z$. Then if $a  \in  x \setminus {z}$, there are two cases.
    If $a  \in  y$, then since $a\not \in y \cap z$, from $y \cap z \subseteq^- y$
    we know that $a$ is negative. If $a \not\in y$, then by $x \cap y \subseteq^- x$ it must be negative.
    Thus $x \supseteq^- (x \cap z)$ as desired -- the other inclusion
    is similar.\qedhere
\end{proof}

If $x \sqsubseteq_A y$ then intuitively $y$ has more output for less input. 
This is analogous to, and in special cases coincides with,   the order on functions in domain theory; hence the name ``Scott order'' . 
In summary, configurations of $\CCC_A \circledast S $ correspond to pairs
$(x, y) \in \mathscr{C} (S) \times \mathscr{C} (A)$ with
$y \sqsubseteq_A \sigma x$. 

%
%
%
%

\subsubsection{Step \#3: Discrete fibrations}

Since configurations of $\CCC_A \circledast S$ can be elegantly expressed using the Scott order,
it will be key to our proof that strategies satisfy a \emph{discrete fibration property} 
with respect to it. We first recall:

\begin{defi}[Discrete fibration]\label{def:disfib}
  Let $(X, \leq_X)$ and $(Y, \leq_Y)$ be orders and
  $f : X \rightarrow Y$ be a monotonic map. It is a \textbf{discrete
    fibration} when for all $x \in X, y \in Y$ such that $y \leq_Y fx$
  there exists a unique $x' \leq_X x \in X$ such that $fx' = y$.
\end{defi}

Now, we prove the following characterisation of courtesy and receptivity.

\begin{lem}\label{lemma:fibration}
  Let $ \sigma  : S  \rightarrow  A$ be a pre-strategy. The following are equivalent:
  \begin{enumerate}[label=(\roman*)]
  \item [(i)] $ \sigma $ is courteous and receptive,
  \item [(ii)] $ \sigma  : ( \mathscr{C} (S), \supseteq^-)  \rightarrow  ( \mathscr{C} (A), \supseteq^-)$ and
    $ \sigma  : ( \mathscr{C} (S), \subseteq^+)  \rightarrow  ( \mathscr{C} (A), \subseteq^+)$ are discrete fibrations,
  \item [(iii)]
    $ \sigma : ( \mathscr{C} (S), \sqsubseteq_S) \rightarrow (
    \mathscr{C} (A), \sqsubseteq_A)$ is a discrete fibration.
  \end{enumerate}
\end{lem}
\begin{proof}
  \noindent (iii) $ \Rightarrow $ (ii):\ straightforward.\medskip

  \noindent (ii) $ \Rightarrow $ (i):\ \emph{Courtesy.} If $s_1^+  \rightarrowtriangle  s_2$ in $S$, then by
    using the discrete fibration property for $\subseteq^+$ we prove
    $ \sigma s_1  \leq   \sigma s_2$ (hence $ \sigma s_1  \rightarrowtriangle   \sigma s_2$ by Lemma \ref{lemma:reflect}). 
    Indeed if it is not the case, then $ \sigma s_1$ and $ \sigma s_2$ are concurrent in $A$ -- otherwise we would have
    $\sigma s_2 \leq \sigma s_1$, so $s_2 \leq s_1$ by Lemma \ref{lemma:reflect}, absurd.

    Hence $ \sigma [s_2] \setminus \{ \sigma s_1\}$ is a configuration of $A$ that positively
    extends to $ \sigma [s_2]$. Thus $[s_2]$ should be the positive extension
    of a configuration $x$ whose image in the game is
    $ \sigma [s_2] \setminus\{ \sigma s_1\}$. By local injectivity, $\sigma s_1 \neq \sigma s_2$,
    therefore $\sigma s_2 \in \sigma [s_2] \setminus\{ \sigma s_1\}$. By local injectivity again,
    this implies that $s_2\in x$, so $s_1\in x$ by down-closure, so $\sigma s_1 \in \sigma [s_2]
\setminus \{ \sigma s_1\}$,
    absurd. 

    If $s_1  \rightarrowtriangle  s_2^-$, the only case not already covered by the above is that
of $s_1^- \imc
    s_2^-$. Assume $\sigma\,s_1$ and $\sigma\,s_2$ are concurrent in $A$. Set $x = [s_2] \setminus
\{s_1,     s_2\} \in \conf{S}$. We have $\sigma x \subseteq^- \sigma x \cup \{\sigma s_2\}$, so by existence of
the discrete fibration property there is $x \subseteq x \cup \{s'_2\} \in \conf{S}$ and $\sigma s'_2 = \sigma s_2$. 
But likewise, $\sigma\,(x\cup \{s'_2\})$ extends in $A$
with $\sigma\,s_1$, so by existence of the discrete fibration property there is $s'_1$
such that $\sigma\,s'_1 = \sigma\,s_1$ and $x\cup \{s'_1, s'_2\} \in \conf{S}$. But then by
uniqueness of the dicrete fibration property we have $x\cup \{s_1, s_2\} = x \cup \{s'_1, s'_2\}$
so by local injectivity $s_1 = s'_1$ and $s_2 = s'_2$, contradicting $s_1 \imc s_2$ since
$x \cup \{s'_2\} \in \conf{S}$.

    \emph{Receptivity.} This is just an instance of the fibration
    property for $\supseteq^-$ for atomic extensions.\medskip

\noindent (i) $ \Rightarrow $ (iii):\ Let $x  \in   \mathscr{C} (S)$ and $y  \in   \mathscr{C} (A)$ such that $y \sqsubseteq  \sigma x$.

    \smallskip

    \emph{Uniqueness.} We prove by induction on the cardinal of
    $y \in \mathscr C(A)$, that for all
    $x_1, x_2 \in \mathscr{C} (S)$, if $x_i \sqsubseteq x$ and
    $ \sigma x_i = y$, then $x_1 = x_2$. Assume the result for all
    $y' \in \mathscr{C} (A)$ strictly smaller than a fixed
    $y \in \mathscr{C} (A)$. 

    First, we prove that $x_1$ and $x_2$ have the same positive
    events. Indeed if $s_1 \in x_1$ is positive, then by
    $\sigma x_1 = y = \sigma x_2$ there is a (unique) $s_2 \in x_2$
    such that $\sigma s_1 = \sigma s_2$. Since $x_i \sqsubseteq x$,
    $s_1$ and $s_2$ are in $x$, and by local injectivity implies $s_1 = s_2$.

    If all maximal events of $x_1$ and $x_2$ are positive, we are
    done by down-closure. Otherwise one of them has a negative maximal
    event, say \emph{wlog.}  $s_1 \in x_1$. Since
    $\sigma x_1 = \sigma x_2$ there is a unique $s_2 \in x_2$ such
    that $\sigma s_1 = \sigma s_2$. 

    If there exists $s_2' \in x_2$ with $s_2 \imc s'_2$, since
    $\sigma s_2$ is maximal in $\sigma x_1 = \sigma x_2$ (from Lemma
    \ref{lemma:reflect}, $ \sigma $ reflects causality), by courtesy we
    must have $s_2'$ positive, and hence $s_2' \in x_1$. It follows
    that $s_1, s_2$ are consistent (both in $x_1$).  Hence $s_1 = s_2$,
    and $s_1 \imc s_2'  \in  x_1$, which is absurd. Therefore $s_2 $ is
    maximal in $x_2$.

    This entails that $x_1 \setminus \{s_1\}$ and $x_2 \setminus \{s_2\}$
    are configurations of $S$ to which we can apply the
    induction hypothesis for the smaller $y' := y \setminus
    \{ \sigma s_1\}$: the configurations $x_1 \setminus\{s_1\}$ and
    $x_2\setminus \{s_2\}$ must be equal. Since
    $ \sigma s_1 = \sigma s_2$ is a negative extension of
    $ \sigma x_1 \setminus\{ \sigma s_1\}$, by receptivity it follows that
    $s_1 = s_2$.

    \bigskip
    \emph{Existence.} By induction on $\sqsubset$, the irreflexive version of $\sqsubseteq$ (by splitting it into atomic extensions).    
    If $y \longcov {a^+}  \sigma x$, write $s$ for the preimage of $a$ in
    $x$. If $s$ is not maximal in $x$, it means that there exists
    $s  \rightarrowtriangle  s'$ in $x$. By courtesy since $s$ is positive, we have
    $ \sigma s  \rightarrowtriangle   \sigma s'$ in $ \sigma x$.  Hence $a$ is not maximal in $ \sigma x$ which is
    absurd.
    If $ \sigma x \longcov {a^-} y$, it is a consequence of receptivity.\qedhere
\end{proof}

Note that for a pre-strategy $ \sigma  : S  \rightarrow  A$ it is not equivalent to be
receptive and to be a discrete fibration
$(\conf S, \supseteq^-) \rightarrow (\conf A, \supseteq^-)$, as
demonstrated by the following pre-strategy on the game $A =  \circleddash _1\  \circleddash _2$:

$$
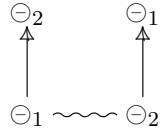
\xymatrix {
  \circleddash _2&
  \circleddash_1 \\
  \circleddash _1\ar@{~}[r] \ar@{-|>}[u] & \circleddash _2\ar@{-|>}[u] }$$
This pre-strategy is receptive but not a discrete fibration for
$\supseteq^-$. Indeed, for $x = \emptyset$,
$y = \{ \circleddash _1, \circleddash _2\}$ there are two possible
matching extensions $x\subseteq x'$. This pre-strategy fails courtesy
-- the equivalence only holds on courteous pre-strategies.

{Putting together the description of configurations of the interaction with copycat of Lemma
\ref{lemma:deadlock}, the characterisation of configurations of copycat in Lemma \ref{lemma:scott}
and the discrete fibration property above, we get
the first direction of the isomorphism between $\conf{S}$ and $\conf{\CCC_A \odot S}$.

\begin{prop}
Let $\sigma : S \to A$ be receptive and courteous. The discrete fibration property yields a
function:
\[
L_\sigma : \conf{\CCC_A \odot S} \to \conf{S}
\]
commuting with the projection to $A$ (\emph{i.e.} for all $x\in \conf{\CCC_A \odot S},
\sigma\,(L_\sigma\,x) = (\ccc_A \odot \sigma)\,x$).
\end{prop}
\begin{proof}
Take $x \in \conf{\CCC_A\odot S}$, yielding $[x]_{\CCC_A \circledast S} \in \conf{\CCC_A \circledast
S}$. In turn, we get
\[
\Psi([x]_{\CCC_A \circledast S}) = (y_S, y_A) \in \conf{S} \times \conf{A}
\]
such that $y_A \sqsubseteq_A \sigma\,y_S$. But then, by the
discrete fibration property, there is a unique $z_S \sqsubseteq_S y_S \in \conf{S}$ such that
$\sigma\,z_S = y_A$; and we set $L_\sigma(x) = z_S$.
\end{proof}

We still have to prove that $L_\sigma$ preserves inclusion. To establish that, we need first to
characterise how the order on $\conf{\CCC_A \odot S}$, and in particular the covering relation, is
reflected on interaction witnesses through $\Psi$.

\begin{lem}\label{lemma:carac_cov}
  Let $x^1, x^2 \in \conf{\CCC_A \odot S}$ and let $(y^1_S, y^1_A), (y^2_S,y^2_A)$ be the
representations of minimal witnesses (\emph{e.g.}, $(y^1_S, y^1_A) = \Psi([x^1]_{\CCC_A\circledast
S})$).
The following are equivalent:
$$
\begin{aligned}
  x^1 \longcov + x^2\text{ in }\CCC_A \odot S&~~\Leftrightarrow~~ y^1_S \subseteq y^2_S~\&~y^1_A
\longcov + y^2_A \\
  x^1 \longcov - x^2\text{ in }\CCC_A \odot S&~~\Leftrightarrow~~ y^1_S = y^2_S~\&~y^1_A \longcov -
y^2_A \\
\end{aligned}$$
\end{lem}
\begin{proof}
\emph{Positive extension.}
Assume $x^1 \longcov + x^2$. Then
  $( \ccc_A \odot \sigma)\,x^1 \longcov + ( \ccc_A \odot \sigma )\,x^2$ implying $y^1_A \cov
  y^2_A$.
  Moreover, we have $[x^1]_{\CCC_A  \circledast  S} \subseteq [x^2]_{\CCC_A  \circledast  S}$
  implying $y^1_S \subseteq y^2_S$.

  Conversely, we have $y^1_A \sqsubseteq_A y^2_A \sqsubseteq_A  \sigma y^2_S$ by hypothesis.  Hence
$(y^2_S, y^1_A)  \in  \Psi(\CCC_A  \circledast  S)$. Writing
 $\subseteq^0$ for extension by invisible events in $\CCC_A  \circledast  S$, we have:

  $$[x^1]_{\CCC_A  \circledast  S} = \Psi^{-1}(y^1_S, y^1_A) \subseteq^0 \Psi^{-1}(y^2_S, y^1_A)
\longcov + \Psi^{-1}(y^2_S, y^2_A) = [x^2]_{\CCC_A  \circledast  S}$$
  Hence $x^1 \longcov + x^2$ as desired.

\emph{Negative extension.}
If $x^1 \longcov - x^2$, then we have $y^1_A \longcov - y^2_A$ and
  $y^1_S \subseteq y^2_S$ by the same argument as in the previous
  equivalence.

  Assume there were a $s \in y^2_S \setminus y^1_S$. Without loss of
  generality $s$ can be assumed maximal in $y^2_S$. By Lemma \ref{lemma:negative_max}, $s$ is
  positive and $ \sigma s  \in  y^2_A$. If we had $\sigma\,s \in y^1_A$, then we would have 
  $\sigma\,s \in \sigma\,y^1_S$ as well as $y^1_A \sqsubseteq_A \sigma\,y^1_S$; so 
  $s\in y^1_S$ (by local injectivity), absurd. So, $s \in y^2_A \setminus y^1_A$, contradicting its
  positivity. Therefore, $y^1_S = y^2_S$ as desired.

  Conversely, if $y^1_S = y^2_S$ and $y^1_A \longcov - y^2_A$ then we have this
  extension in $\CCC_A  \circledast  S$:
  $$[x^1]_{\CCC_A  \circledast  S} = \Psi^{-1}(y^1_S, y^1_A) \longcov - \Psi^{-1}(y^1_S, y^2_A) =
  [x^2]_{\CCC_A  \circledast  S}$$
  yielding $x^1 \longcov - x^2$ in $\CCC_A \odot S$, since the event we added is visible.
\end{proof}

From that we easily get:

\begin{prop}
Let, $\sigma : S \to A$ be receptive and courteous, then the function 
$L_\sigma : \conf{\CCC_A \odot S} \to \conf{S}$ is monotonic.
\end{prop}
\begin{proof}
We prove it for coverings. If $x^1, x^2 \in \conf{S}$ are such that 
$x^1 \cov x^2$, we write $(y^1_S, y^1_A) = \Psi([x^1]_{\CCC_A\circledast S})$ and $(y^2_S,
y^2_A) = \Psi([x^2]_{\CCC_A\circledast S})$ for the representations of their minimal interaction witnesses. 

We distinguish two cases, depending on the polarity of the extension. If $x^1 \longcov{-} x^2$, then 
by the lemma above $y^1_S = y^2_S$ and $y^1_A \longcov{-} y^2_A$. It immediately follows that
$L_\sigma(x^1) \longcov{-} L_\sigma(x^2)$ by uniqueness of the discrete fibration property.
If $x^1 \longcov{+} x^2$, then $y^1_S \subseteq y^2_S$ and $y^1_A \longcov{+} y^2_A$. 
But then by
Lemma \ref{lemma:negative_max} we actually have $\sigma\,y^1_S \subseteq^- y^1_A$ and $\sigma\,y^2_S
\subseteq^- y^2_A$, from which it follows that $y^1_S \longcov{+} y^2_S$ as well. And then again, 
$L_\sigma(x^1) \longcov{+} L_\sigma(x^2)$ follows directly from uniqueness of the discrete fibration
property.
\end{proof}
}

\subsubsection{Step \#4: Reconstructing minimal interactions}
Reciprocally, from $x\in \conf{S}$, we have seen that
the pair $(x, \sigma\,x)$ represents a configuration in $\CCC_A \circledast S$ that 
gives us a configuration of the composition through hiding. But it might not be the
minimal witness, \emph{i.e.} it might not satisfy the conditions of Lemma \ref{lemma:negative_max}.

In order to prove the desired isomorphism, we need to extract from $x$ a $x^*$ such that
$(x^*, \sigma\,x)$ satisfies these conditions.
The configuration $x^*$ 
 is obtained by stripping all the maximal negative events away from  $x$, as detailed now.

\begin{lem}\label{lemma:unique_antecedant}
  Let $x  \in   \mathscr{C} (S)$. There is a unique $x^* \subseteq x  \in   \mathscr{C} (S)$ such that $\Psi^{-1} (x^*,  \sigma x)   \in  \mathscr{C} (\CCC_A  \circledast  S)$
  and all maximal events of $\Psi^{-1} (x^*,  \sigma x)$ are visible. 
{Restricting to visible
events, this 
yields a monotonic function
\[
\begin{array}{rcrcl}
R_\sigma &:& \conf{S} &\to& \conf{\CCC_A \odot S}\\
&& x &\mapsto& (\Psi^{-1}(x^*, \sigma x)) \cap (\CCC_A \odot S)
\end{array}
\]
}\end{lem}
\begin{proof}
  \emph{Uniqueness}. Assume we have two $x'_1$ and $x'_2$ in
  $ \mathscr{C} (S)$ satisfying the hypotheses. The configurations
  $\Psi^{-1} (x'_1,  \sigma x)   \in  \mathscr{C} (\CCC_A  \circledast  S)$ and $\Psi^{-1} (x'_2,  \sigma x)   \in  \mathscr{C} (\CCC_A  \circledast  S)$
  correspond to secured bijections:
  \[
  x'_1 \parallel \sigma x \stackrel{\varphi_1}{\simeq} \sigma x'_1 \parallel \sigma x ~~~~~~~~~~~~~~~~~~ x'_2 \parallel \sigma x \stackrel{\varphi_2}{\simeq} \sigma x'_2 \parallel \sigma x
  \]
  whose maximal events are visible, and $\sigma x \sqsubseteq \sigma x'_1, \sigma x \sqsubseteq \sigma x'_2$.

  By Lemma \ref{lemma:negative_max}, the maximal events of $x'_1$ and
  $x'_2$ are positive. Moreover, we have
  $\sigma x'_1 \subseteq^- \sigma x$. Indeed, we already know that
  $\sigma x'_1 \supseteq^+\cdot  \subseteq^- \sigma x$, and for
  $a^+ \in \sigma x$, we have $(0, a) \leq (1, a) \in \CCC_A$. So,
  there is $((0, s), (0, a))\in \varphi_1$.  Therefore,
  $a = \sigma s \in \sigma x'_1$. With these two remarks, it is
  elementary to check (using $x'_1 \subseteq x$ and local injectivity)
  that $x'_1 = [x^+]$, where $x^+$ denotes the set of positive events
  of $x$ -- the same reasoning holds for $x'_2$, and hence $x'_1 = x'_2$.

%
%
%
  \emph{Existence}. Write $x^* = [x^+]_S$. The set $x \setminus x^*$
  contains all the negative events of $x$ without any positive event
  above them, thus we have $x^* \subseteq^- x$. Thus
  $ \sigma x \sqsubseteq \sigma x^*$, therefore
  $\Psi^{-1}(x^*, \sigma x) \in \conf{\CCC_A\circledast S}$.  Maximal
  events are visible because $x^*$ and $ \sigma x$ satisfy the
  condition (ii) of Lemma \ref{lemma:negative_max}.

{From the definition of $(-)^*$ above, the monotonicity of $R_\sigma$ is clear.}
\end{proof}

\subsubsection{Step \#5: Characterising the Scott order on $\conf{\CCC_A\odot S}$}

Using the above, we can prove that indeed receptivity and courtesy are sufficient to
be preserved by composition with copycat (proof forthcoming in Theorem \ref{thm:main_thm}).
For necessity, we will prove that strategies
obtained by composition with copycat are automatically discrete fibrations. In order to do that, we
first need to study the Scott order on $\conf{\CCC_A \odot S}$
(we write $V$ for the set of visible events of $\CCC_A \circledast S$, that is, the events of $\CCC_A \odot S$).

As we have seen, configurations of $\CCC_A \odot S$ correspond to certain 
pairs $\Psi(z) = (x, y) \in \conf{S} \times \conf{A}$ where the maximal events of $x$ are positive.
Progressing in $\sqsubseteq_{\CCC_A \odot S}$ means removing
some (maximal) negative events from $y$, and adding some positives to it. The first part is easy, as
these events had not been propagated to $x$ yet. However, adding some positives in $y$
might require to replay them first in $x$, along with their negative dependencies. For instance:

\begin{exa}
Consider $A = \mathbb W_1 \parallel \mathbb W_2$ and $\sigma$ playing on $A$, with event
structure $\mathbf \Click_1^- \imc \mathbf \Compute_1^+$ and concurrent $\mathbf \Click_2^-$. 
The two interactions below are minimal witnesses of (respectively) $x_1, x_2 \in \conf{\CCC_A \odot S}$, 
with $x_1 \sqsubseteq_{\CCC_A \odot S} x_2$:
\[
\xymatrix@R=5pt@C=5pt{
\mathbb W_1
	\ar@{}[r]|\parallel&
\mathbb W_2&
\mathbb W_1
	\ar@{}[r]|\parallel&
\mathbb W_2&~~~~~~~~~~&
\mathbb W_1
        \ar@{}[r]|\parallel&
\mathbb W_2&
\mathbb W_1 
        \ar@{}[r]|\parallel&
\mathbb W_2\\
&&\mathbf \Click_1^-&
\mathbf \Click_2^-&&&&
\mathbf \Click_1^-
	\ar@{-|>}[dll]\\
&&&&\sqsubseteq&\mathbf \Click_1
	\ar@{-|>}[d]\\
&&&&&\mathbf \Compute_1
	\ar@{-|>}[drr]\\
&&&&&&&\mathbf \Compute_1^+
}
\]
We observe that although the visible part progresses \emph{w.r.t.} the Scott order, the invisible
part only gains events, and potentially of both polarities: it progresses \emph{w.r.t.} plain inclusion.
\end{exa}

Formally, we prove the following lemma.

\begin{lem}\label{lemma:sc_on_comp}\label{lemma:scott_phi}
  Let $z, z' \in \conf{\CCC_A \odot S}$ and let $(x, y), (x',y')$ be the respective representations
of their minimal witnesses via $\Psi$.
The following are equivalent:
  \begin{enumerate}
  \item $z \sqsubseteq_{\CCC_A \odot S} z'$
  \item $y \sqsubseteq_A y'$ and $x \subseteq x'$
  \end{enumerate}
\end{lem}
\begin{proof}
{Immediate consequence of Lemma \ref{lemma:carac_cov}.}
\end{proof}

\subsubsection{Step \#6: Wrapping up}
\label{subsubsec:wrapup}

Having introduced all the tools and lemmas needed for our proof, we now prove the main theorem.

\begin{thm}\label{thm:main_thm}
  Let $ \sigma  : S  \rightarrow  A$ be a pre-strategy. The following are equivalent:
  \begin{enumerate}[label=(\roman*)]
  \item [(i)] $ \sigma $ is a strategy
  \item [(ii)] $ \sigma  : ( \mathscr{C} (S), \supseteq^-)  \rightarrow  ( \mathscr{C} (A), \supseteq^-)$ and
    $ \sigma  : ( \mathscr{C} (S), \subseteq^+)  \rightarrow  ( \mathscr{C} (A), \subseteq^+)$ are discrete fibrations,
  \item [(iii)] the map $ \sigma  :( \mathscr{C} (S), \sqsubseteq_S)  \rightarrow  ( \mathscr{C} (A), \sqsubseteq_A)$ is a discrete fibration
  \item [(iv)] $ \sigma $ is courteous and receptive
  \end{enumerate}
\end{thm}
\begin{proof}
  The equivalence between (ii), (iii), (iv) is proved by Lemma \ref{lemma:fibration}.\medskip

  \noindent
    (i) $\Rightarrow$ (iii):\ Let $f :  \sigma  \cong  \ccc_A \odot \sigma$ be an
    isomorphism of strategies. Let $(x, y) \in  \mathscr{C} (S)  \times   \mathscr{C} (A)$ with $y \sqsubseteq  \sigma x$. Write $\Psi([f(x)]_{\CCC_A  \circledast  S}) = (w,  \sigma x)  \in \Psi(\CCC_A  \circledast  S) $
    with $w  \in   \mathscr{C} (S)$ and $ \sigma x \sqsubseteq  \sigma w$.

    \emph{Existence.} Consider $x_0 = [\{ s  \in  w \mid  \sigma s  \in  y\}]^*$ (Lemma
    \ref{lemma:unique_antecedant}). By definition the maximal events
    of $\Psi^{-1}(x_0, y)$ are all visible.  Hence $(x_0, y)$
    corresponds to a configuration $z  \in   \mathscr{C} (\CCC_A \odot S)$. Applying
    $f^{-1}$ we get a configuration $x' \in \mathscr{C} (S)$ whose
    image by $\sigma$ is $y$.  Since $y \sqsubseteq \sigma x$ and
    $x_0 \subseteq w$, we have by Lemma
    \ref{lemma:scott_phi}, $z \sqsubseteq f(x)$.  Hence
    $x' \sqsubseteq x$ ($f^{-1}$ preserves the Scott order).

    \emph{Uniqueness.} Assume we have two $x'_1$ and $x'_2$ satisfying
    $x'_i \sqsubseteq x$ and $ \sigma x'_i = y$. We have
    $f(x'_1) = V \cap (\Psi(x''_1, y))$ and
    $f(x'_2) = V \cap (\Psi(x''_2, y))$ for some configurations $x''_1$
    and $x''_2$. Applying Lemma \ref{lemma:unique_antecedant} we get
    $x''_1 = x''_2$ which yields $f(x'_1) = f(x'_2)$ and then
    $x'_1 = x_2'$ by injectivity of $f$.\medskip

  \noindent (iii) $\Rightarrow$ (i):\ {We have constructed two inclusion-preserving maps
$L_\sigma : \conf{\CCC_A \odot S} \to \conf{S}$ and $R_\sigma : \conf{S} \to \conf{\CCC_A \odot S}$.
By construction, they are inverses -- $L_\sigma \circ R_\sigma = \id_{\conf{S}}$ by uniqueness of
the discrete fibration property, and $R_\sigma \circ L_\sigma = \id_{\conf{\CCC_A \odot S}}$ by
uniqueness of the $(-)^*$ operation in Lemma \ref{lemma:unique_antecedant}. By Lemma
\ref{lemma:eviso}, this yields the desired isomorphism between $\ccc_A \odot \sigma$ and $\sigma$.}\qedhere
\end{proof}

\section{The bicategory of concurrent games}
\label{sec:bicategory}
We have developed a notion of concurrent strategies, and characterised those
which behave well in an asynchronous, distributed world. 
For this to serve as a basis for the compositional semantics of 
concurrent processes or programs, it is of paramount importance to study the
categorical structure of strategies, \emph{i.e.} the algebraic laws
satisfied by composition. 

Usually -- as described first by Joyal on Conway games \cite{JoyalGazette} -- composition
of strategies yields a category having games as objects, strategies as morphisms
and copycat strategies as identities. Here however, we cannot use equality to
compare strategies. Indeed, take $\sigma : S \to A$ and $\sigma' : S' \to A$
two strategies on $A$. As we have observed in Section \ref{sec:evstrat}, comparing
them requires us first to relate $S$ and $S'$, which we do via a map
$f : S \to S'$ making the obvious triangle commute. This map is in general not
unique: {we saw below Definition \ref{def:strategy} a strategy with two automorphisms.}

{For 
many purposes, the exact identity of an isomorphism 
relating two strategies is 
irrelevant, and in these cases we can (and 
later will) quotient to a category.
This quotient, and the investigation of its further structure, will be carried out in Section
\ref{sec:compact}. But
the un-quotiented structure also matters -- when working with our games,
one is 
often led 
to reason on representatives rather than isomorphism classes (for instance
when computing infinite strategies as limits of $\omega$-chains of finite strategies).
Similarly, further developments in this framework (
beyond this paper) rely on
properties of composition that the quotiented category is too rough to convey.
So we first investigate the composition operation without quotienting, and
show
how the specific isomorphisms}
between strategies fit in the categorical picture. This is the purpose of this
section, where we will establish that games, strategies and maps between them
form a bicategory. We will {first review the definition of a bicategory. Then, we will turn to
the construction of our concrete bicategory of concurrent games.}

\subsection{Bicategories}
{First, recall that a \textbf{bicategory} $\C$ consists of the following \emph{basic data} 
(with notations inspired from our concrete bicategory of concurrent games):
\begin{itemize}
\item A set of \emph{objects}, or $0$-cells {(we use $A, B, C, \dots$ to range over objects)}.
\item For any two objects $A$, $B$, a category $\C(A,B)$. Its
objects are the \emph{morphisms} or $1$-cells of $\C$ (we use $\sigma, \tau, \dots$ to range over
morphisms, and write \emph{e.g.} $\sigma : A \profto B$), and its morphisms are the \emph{$2$-cells} of $\C$ (we
range $f, g, \dots$ to range over $2$-cells, and write \emph{e.g.} $f : \sigma \tto \tau$).
\item For each object $A$, a distinguished morphism $\ccc_A : A \profto A$, called the
\emph{identity}.
\item For each objects $A, B, C$, a functor:
\[
\odot : \C(B, C) \times \C(A, B) \to \C(A, C)
\]
\end{itemize}

\noindent The functor $\odot$ gives the \emph{composition} $\tau \odot \sigma : A \profto C$ of $1$-cells
$\sigma : A \profto B$ and $\tau : B \profto C$, but its functorial action also allows us to transport
$2$-cells alongside compositions. For instance, if $f : \sigma \tto \sigma'$, then $\tau \odot f :
\tau \odot \sigma \tto \tau \odot \sigma'$ (A consequence of that is that isomorphism of $1$-cells
is a congruence, \emph{i.e.} is preserved under composition).

But that is not all. In a bicategory, the associativity of $\odot$ and neutrality of $\ccc_A$ do not in general hold
in the strict sense (or we would have a $2$-category), but only up to \emph{coherent isomorphisms}.
This means that we have the following \emph{isomorphisms}:
\begin{itemize}
\item For any $\sigma : A \profto B, \tau : B \profto C, \rho : C \profto D$, an isomorphism (the
\emph{associator}):
\[
\alpha_{\sigma, \tau, \rho} : (\rho\odot \tau) \odot \sigma \tto \rho\odot (\tau \odot \sigma)
\]
natural in $\sigma, \rho$ and $\tau$.
\item For any $\sigma : A \profto B$, two isomorphisms (the \emph{unitors}):
\[
\rho_\sigma : \sigma \odot \ccc_A \tto \sigma \hspace{40pt} \lambda_\sigma : \ccc_B \odot \sigma \tto
\sigma
\]
natural in $\sigma$.
\end{itemize}

\noindent Finally, these data need to satisfy some coherence conditions: the associators are subject to Mac
Lane's \emph{pentagon identity} whereas the unitors must satisfy the \emph{triangle identity} (those are
the same as in a monoidal category). We do not recall them now, but we will state them in the course
of the construction of our concrete bicategory.

We now go on to construct our concrete bicategory. We gave the definition of a bicategory in two
steps: first what we called the \emph{basic data}, and then natural isomorphisms for associativity
and unities, subject to coherence conditions. Our construction of the concrete bicategory will follow
the same lines.}

\subsection{{Basic data of the bicategory $\CG$}} {As expected, the \emph{objects} of $\CG$ are the
games, and the \emph{morphisms} from $A$ to $B$ are the strategies $\sigma : S \to A^\perp \parallel
B$. As in the definition above we will occasionally write $\sigma : A \profto B$, keeping the $S$
anonymous.

For $\sigma : S \to A^\perp \parallel B$, $\tau : T \to B^\perp \parallel C$, a \emph{$2$-cell} $f :
\sigma \tto \tau$ is a map of esps $f : S \to T$ making the following triangle commute:
\[
\xymatrix@R=15pt@C=15pt{
S       \ar@/^/[rr]^f
        \ar[dr]_{\sigma}&&
T       \ar[dl]^{\tau}\\
&A^\perp \parallel B
}
\]
Such $2$-cells can be composed as functions, and for $\sigma : S \to A^\perp \parallel B$ the
identity $\id_S : S \to S$ is a valid $2$-cell $\id_\sigma : \sigma \tto \sigma$.
Therefore, for any two games $A$ and $B$ we get a category $\CG(A, B)$ as required by the definition.}

\paragraph{Functorial composition.} {Now, we need a functor:}
\[
\odot : \CG(B, C) \times \CG(A, B) \to \CG(A, C)
\]
For $\tau : B \profto C$ and $\sigma : A \profto B$, its action is 
the composition $\tau \odot \sigma$ as in Section \ref{sec:evstrat}.
This operation was defined on pre-strategies rather than strategies, so we note in 
passing:

\begin{prop}
For $\sigma : S \to A^\perp \parallel B$ and $\tau : T \to B^\perp \parallel C$ strategies, $\tau \odot \sigma$
is a strategy.
\end{prop}
\begin{proof}
We use the second formulation of the definition of strategies, as in Theorem \ref{thm:main_thm}.

\emph{Negative fibration.} Take $x\in \conf{T \odot S}$ such that
$(\tau \odot \sigma)(x) \subseteq^- x'_A \parallel x'_C$ for some $x'_A 
\parallel x'_C \in \conf{A^\perp \parallel C}$.
By definition, its down-closure in $T\circledast S$ is a configuration
$y = [x] \in \conf{T\circledast S}$, whose maximal elements are
visible. By Lemma \ref{lemma:confinter}, this configuration is
represented by (the graph of) a secured bijection
$\varphi \in \secbij{\sigma\parallel C, A\parallel \tau}$.  We
write:
\[
y_S \parallel y_C~~\stackrel{\varphi}{\simeq}~~y_A \parallel y_T
\]
with $\sigma y_S = y_A \parallel y_B$ and $\tau y_T = y_B \parallel y_C$. 
By hypothesis we have $y_A \parallel y_B \subseteq^-y'_A \parallel y_B$, and 
$y_B \parallel y_C \subseteq^- y_B \parallel y'_C$ for some $y'_A \in 
\conf{A^\perp}$ and $y'_C \in \conf{C}$.
Since $\sigma$ and $\tau$ are strategies, there are unique
$y_S \subseteq y'_S \in \conf{S}$ and $y_T \subseteq y'_T \in \conf{T}$ such 
that $\sigma y'_S = y'_A \parallel y_B$ and
$\tau y'_T = y_B \parallel y'_C$. The induced extension of $\varphi$
\[
y'_S \parallel y'_C~~\stackrel{\varphi'}{\simeq}~~y'_A \parallel y'_T
\]
is secured: the added events only map to $A$ and $C$, so there is no interaction (hence potential deadlock) between $\sigma$ and $\tau$ going on. 
Moreover, $\varphi'$ represents a configuration $y\subseteq y'\in \conf{T\circledast S}$,
which maps to $x_A' \parallel x_B \parallel x'_C$. By projection we get the required extension of $x$. Uniqueness
follows directly from uniqueness for $y'_S$ and $y'_T$.

\emph{Positive fibration.} Similar reasoning.
\end{proof}

So composition, despite being defined on pre-strategies rather than strategies, preserves courtesy and
receptivity -- it is well-defined on $1$-cells of our bicategory. We now need to prove that it is
well-defined on \emph{$2$-cells} as well. In fact, we will show that it is well-defined on morphisms
between arbitrary pre-strategies, not only those that are receptive and courteous. 
Until Section \ref{subsec:unitors} (where we study compositions with copycat), the development 
will use neither receptivity nor courtesy.

Let
$\sigma : S \to A^\perp \parallel B$, $\sigma' : S' \to A^\perp \parallel B$ and $\tau : T \to
B^\perp \parallel C$ be pre-strategies, and $f : S \to S'$ be a morphism from $\sigma$
to $\sigma'$. We proved in Lemma \ref{lem:pullback} that the interaction $T\circledast S$ was the pullback of
$\sigma \parallel C$ and $A \parallel \tau$. By the corresponding universal property, it follows that
there is a unique map $f \circledast T : S \circledast T \to S'\circledast T$ making the required diagrams commute. In
particular, this remark establishes that the interaction operation $- \circledast -$ is functorial
in morphisms between pre-strategies.
In order for $\odot$ to inherit this, it is convenient to use that $\circledast$ and $\odot$ are related by a universal
property involving \emph{partial maps}:

\begin{defi}\label{def:partialmap}
A \textbf{partial map} of es(p)s $f : E \pto F$ is a partial function, such that
for all $x \in \conf{E}$ we have $f x \in \conf{F}$, and such that for all $e_1, e_2 \in x \in \conf{E}$,
if $f e_1 = f e_2$ (with both defined), then $e_1 = e_2$.
\end{defi}

A key example of a partial map in our setting, is the \emph{hiding map}: given an es(p) $E$ and
$V \subseteq E$, there is a partial map:
\[
\hid : E \pto E \proj V
\]
acting as the identity on $V$ and undefined otherwise. So in particular, for
pre-strategies $\sigma : S \to A^\perp \parallel B$ and $\tau : T \to B^\perp \parallel C$, there is a partial map:
\[
\hid : T \inter S \pto T \odot S.
\]
Projection and hiding provide a partial-total factorization system, which obeys:

\begin{lem}
Let $f : E \pto F$ be a partial map of es(p)s, and $V$ be the subset of events of $E$ on which $f$ is defined. Then, $f$ factors as
$(f \upharpoonright V) \circ \hid$ (where $f \upharpoonright V : E \proj V \to F$ is total). Moreover, for any other
factorisation $f = g_2 \circ g_1$ with $g_1 : E \pto X$ and $g_2 : X \to F$, there is a unique total $h : E\proj V \to X$ such that
$h \circ \hid = g_1$ and $g_2 \circ h = f \upharpoonright V$, as pictured in the diagram below:
\[
\xymatrix{
E       \ar@^{->}[d]^{\hid}
        \ar@^{->}[dr]^{g_1}
        \ar@^{->}@/_2pc/[dd]_f\\
E\proj V\ar[d]^{f\upharpoonright V}
        \ar@{.>}[r]^{h}&
X       \ar[dl]^{g_2}\\
F
}
\]
We say that $\hid : E \pto E \proj V$ has the \textbf{partial-total universal property}.
\label{lem:partotup}
\end{lem}
\begin{proof}
Direct verification.
\end{proof}

From that, it is easy to construct the functorial action of $\odot$. 
Take $\sigma, \sigma', \tau$ and $f$ as above. As explained, we obtain
$T\circledast f : T \circledast S \to T \circledast S'$ by the universal property
of the interaction pullback.

But by Lemma \ref{lem:partotup}, the two maps $\hid_{\sigma, \tau} : T\circledast S \pto T\odot S$ and
$\hid_{\sigma', \tau} : T\circledast S' \pto T\odot S'$ have the partial-total 
universal property. Using it, we get a unique map $T \odot f : T \odot S \to T \odot S'$ matching
$T \circledast f$ up to hiding. It is straightforward from the universal properties that this operation is functorial,
that its symmetric counterparts $g \circledast S$ and $g \odot S$ are as well and that they satisfy the
interchange laws, yielding the required bifunctor. 

In fact we note in passing that $\odot$ preserves more general notions of morphisms of pre-strategies, 
that do \emph{not} leave the game invariant:

\begin{lem}\label{lem:gen_odot_fonc}
Consider two commuting diagrams between pre-strategies (using the obvious functorial action of
$(-)^\perp$ and $- \parallel -$ in $\ESP$):
\[
\xymatrix{
S_1     \ar[r]^{f}
        \ar[d]_{\sigma_1}&
S_2     \ar[d]^{\sigma_2}\\
A_1^\perp \parallel B_1
        \ar[r]^{h_1^\perp \parallel h_2}&
A_2^\perp \parallel B_2
}
~~~~~~
\xymatrix{
T_1     \ar[r]^g
        \ar[d]_{\tau_1}&
T_2     \ar[d]^{\tau_2}\\
B_1^\perp \parallel C_1
        \ar[r]^{h_2^\perp \parallel h_3}&
B_2^\perp \parallel C_2
}
\]
Then, the following diagram commutes.
\[
\xymatrix{
T_1\odot S_1
        \ar[r]^{g \odot f}
        \ar[d]_{\tau_1 \odot \sigma_1}&
T_2\odot S_2
        \ar[d]^{\tau_2\odot \sigma_2}\\
A_1^\perp \parallel C_1
        \ar[r]^{h_1^\perp \parallel h_3}&
A_2^\perp \parallel C_2
}
\]
\end{lem}
\begin{proof}
For interactions first, the map $g\circledast f : T_1 \circledast S_1 \to T_2 \circledast S_2$ is defined from
the universal property of the pullback for $T_2 \circledast S_2$, using the two commuting diagrams
in
the hypothesis. It follows by definition that the diagram
\[
\xymatrix{
T_1\circledast S_1
        \ar[r]^{g \circledast f}
        \ar[d]_{\tau_1 \circledast \sigma_1}&
T_2\circledast S_2
        \ar[d]^{\tau_2\circledast \sigma_2}\\
A_1 \parallel B_1 \parallel C_1
        \ar[r]^{h_1 \parallel h_2 \parallel h_3}&
A_2 \parallel B_2 \parallel C_2
}
\]
commutes. The map $g \odot f : T_1\odot S_1 \to T_2\odot S_2$ and the required diagram commutation
follow from the partial-total universal property.
\end{proof}

\subsection{Associators}
\label{sec:assoc}
{
We now define the \emph{associator}, \emph{i.e.} for every three strategies 
$\sigma : A \profto B, \tau : B \profto C, \rho : C \profto D$, an isomorphism
\[
\alpha_{\sigma, \tau, \rho} : (\rho \odot \tau) \odot \sigma \Rightarrow \rho \odot (\tau \odot \sigma)
\]
natural in $\sigma, \tau, \rho$, and subject to \emph{Mac Lane's pentagon} (detailed in the
development below).}
%
We will start with the definition of the associator.

\paragraph{Associativity for interaction.} For the rest of this subsection we only consider polarity-agnostic operations, so we will
ignore polarity from now on. 

Consider $\sigma : S \to A \parallel B, \tau : T \to B \parallel C$, and $\rho : U \to C \parallel D$. The composition
$\rho \odot \tau : U\odot T \to B\parallel D$ is obtained by restriction from the mediating map $\rho \circledast \tau : U \circledast T \rightarrow B \parallel C \parallel D$
of the interaction pullback. In turn, we can form $(\rho \circledast \tau) \circledast \sigma : (U\circledast T) \circledast S \to A \parallel B \parallel C \parallel D$
as (the mediating map of) the pullback of $\sigma \parallel C \parallel D$ and $A\parallel (\rho \circledast \tau)$. From that (using that pullbacks
are stable under parallel composition) it appears that $(\rho \circledast \tau) \circledast \sigma$ is (the mediating map of) a ternary pullback
of $\sigma \parallel C \parallel D$, $A\parallel \tau \parallel D$
and $A \parallel B \parallel \rho$. But a similar reasoning holds for $\rho \circledast (\tau \circledast \sigma)$, so by the universal property of pullbacks,
there is a unique map $a_{\sigma, \tau, \rho}$, necessarily an isomorphism, making the projections to $\sigma \parallel C \parallel D, A\parallel \tau \parallel D$
and $A\parallel B \parallel \rho$ commute:
\[
\xymatrix@R=10pt@C=10pt{
(U\circledast T)\circledast S
	\ar@/^/[rr]^{a_{\sigma, \tau, \rho}}
	\ar[dr]_{(\rho\circledast \tau) \circledast \sigma\quad}&&
U \circledast (T\circledast S)
	\ar[dl]^{\quad\rho \circledast (\tau \circledast \sigma)}\\
&A\parallel B \parallel C\parallel D
}
\]
Given another $\delta : V \to D \parallel E$, all bracketings of the quaternary interaction between $\sigma, \tau, \rho, \delta$ can 
be obtained via pullbacks of $\sigma \parallel C \parallel D \parallel E, A \parallel \tau \parallel D \parallel E, A\parallel B \parallel \rho \parallel E$
and $A\parallel B \parallel C \parallel \delta$ taken in different orders. It follows from an easy diagram chase that \emph{Mac Lane's pentagon} commutes
at the level of interactions:
\[
\xymatrix@R=10pt@C=10pt{
&((V\circledast U) \circledast T) \circledast S
        \ar[dl]_{a_{\tau, \rho, \delta}\circledast S\quad}
        \ar[ddr]^{a_{\sigma, \tau, \rho\circledast \delta}}\\
(V\circledast (U\circledast T)) \circledast S
        \ar[dd]^{\ \ a_{\sigma, \rho\circledast \tau, \delta}}\\
&&(V\circledast U) \circledast (T\circledast S)
        \ar[ddl]^{\ \ a_{\tau \circledast \sigma, \rho, \delta}}\\
V\circledast ((U\circledast T) \circledast S)
        \ar[dr]_{V\circledast a_{\sigma, \tau, \rho}\quad}\\
&V\circledast (U \circledast (T\circledast S))
}
\]
To conclude associativity, we need to show how to reproduce the same reasoning on composition, or 
more adequately deduce it from that on interactions.

\paragraph{Partial-total factorization and hiding witnesses.} In order to deduce associators on composition and their
coherence from those on interactions, we generalize the partial-total universal property of 
Lemma \ref{lem:partotup} to $n$-ary interactions and compositions. For instance, we need to prove
that the hiding map (to be defined precisely):
\[
\hid : (U\circledast T)\circledast S \pto (U\odot T) \odot S
\]
has the partial-total universal property. It is 
rather inconvenient to prove it directly -- instead,
we prove an auxiliary property that is easier to combine.

\begin{defi}
Let $f : E \pto F$ be a partial map. A \textbf{hiding witness} for $f$ is a monotonic function:
\[
\wit_f : \conf{F} \to \conf{E}
\]
such that for all $x \in \conf{E}$, $\wit_f \circ f(x) \subseteq x$ and for all
$x\in \conf{F}$, $f\circ \wit_f(x) = x$.
\end{defi}

The hiding witness assigns, to any $x \in \conf{F}$, a canonical witness $\wit_f(x) \in \conf{E}$,
that projects back to $x$ through $f$. The hiding witnesses give a configuration-based version
of projection -- or of the partial-total factorization, as established by the lemma below.

\begin{prop}
Let $f : E \pto F$ be a partial map.
Then, the three following propositions are equivalent:
\begin{enumerate}[label=(\roman*)]
\item[(i)] There exists an isomorphism $\varphi : E \proj V \iso F$ such that 
$\varphi \circ \hid = f$ (where $V$ is the domain of definition of $f$ {-- note that $\varphi$
is necessarily $f$ restricted to $V$}),
\item[(ii)] $f$ has the partial-total universal property,
\item[(iii)] $f$ has a hiding witness.
\end{enumerate}
We call \textbf{hiding maps} any partial maps satisfying those properties.
Note that by \emph{(i)} it follows that in any hiding map $f$ is partial rigid, 
\emph{i.e.}
for any $e_1 \leq e_2$, if $f(e_1), f(e_2)$ defined then $f(e_1) \leq f(e_2)$.
\label{prop:hiding_carac}
\end{prop}
\begin{proof}
{(i) $\Leftrightarrow$ (ii).}\ From left to right, we transport through $\varphi$ the partial-total
universal property of Lemma \ref{lem:partotup}. From right to left, we use the fact that both
$\hid : E \pto E\proj V$ and $f : E \pto F$ have the partial-total universal property, yielding
the desired isomorphism.\medskip

\noindent{(ii) $\Rightarrow$ (iii).}\ W.l.o.g., we prove it for $\hid : E \pto E\proj V$. For $x \in \conf{E \proj V}$,
define $\wit(x) = [x] \in \conf{E}$. Clearly, $\hid(\wit(x)) = [x] \cap V = x$ and
$\wit(\hid(x)) = [x\cap V] \subseteq x$ as required, and it preserves union by definition.\medskip

\noindent{(iii) $\Rightarrow$ (i).}\ We construct the isomorphism on configurations:
\[
\begin{array}{rcrcl}
p &:& \conf{E\proj V} &\to& \conf{F}\\
&& x &\mapsto& f([x])\\\\
q &:& \conf{F} &\to& \conf{E\proj V}\\
&& y &\mapsto& \wit(y) \cap V
\end{array}
\]
It is clear by definition that these maps are monotonic, we need to prove that they are inverses of each other.
For one direction, for all $y\in \conf{F}$, since $\wit(y) \in \conf{E}$ it is down-closed in $E$ and thus can only differ
from $[\wit(y)\cap V] \in \conf{E}$ with events not in $V$, so $f([\wit(y)\cap V]) = f(\wit(y)) = y$, \emph{i.e.} $p\circ q(y) = y$.

For the other direction, we note first that if $x\in \conf{E}$ has all its maximal events in $V$, then $\wit(f(x)) = x$. Indeed,
we have $\wit(f(x)) \subseteq x$ by hypothesis. But both sides map to $f(x)$ via $f$, inducing by local injectivity
bijections $\wit(f(x)) \cap V \simeq f(x)$ and $x\cap V \simeq f(x)$. It follows that $\wit(f(x)) \cap V = x \cap V$. But
$x = [x\cap V]$ since its maximal elements are visible. Putting everything together:
\[
x = [x\cap V] = [\wit(f(x)) \cap V] \subseteq \wit(f(x)) \subseteq x
\]
So $x = \wit(f(x))$. Turning back to our main proof, we need to show that $q\circ p(x) = x$ for $x\in \conf{E\proj V}$,
\emph{i.e.} that $\wit(f([x])) \cap V = x$. But by definition, $[x]$ has its maximal events in $V$, so
$\wit(f([x])) = [x]$. So we are left to prove that $[x]\cap V = x$, which is clear.

So we have constructed an order-isomorphism between the domains of configurations of $E\proj V$ and $F$, which
yields an isomorphism by Lemma \ref{lemma:eviso}.
Finally, the required equality is obvious by Lemma \ref{lem:faithful}.
\end{proof}

\paragraph{Associators for composition.} The third formulation of hiding maps enables us to combine them in 
several ways. Firstly, they are stable under composition:

\begin{lem}
Let $\hid : E_1 \pto E_2$ and $\hid' : E_2 \pto E_3$ be hiding maps, then $\hid'\circ \hid : E_1 \pto E_3$ is a hiding map as well.
\label{lem:hiding_composition}
\end{lem}
\begin{proof}
Obvious, by composing the hiding witnesses.
\end{proof}

We can also combine hiding maps ``horizontally", using the universal property of the interaction. For that though, we
need first to prove that this universal property applies to partial maps.

\begin{lem}
A pullback square in $\ES$ is also a pullback square in the category $\ES_\bot$ having event structures as objects, and
\emph{partial maps} as morphisms.
\label{lem:partial_pb}
\end{lem}
\begin{proof}
The proof is summarized in the following diagram:
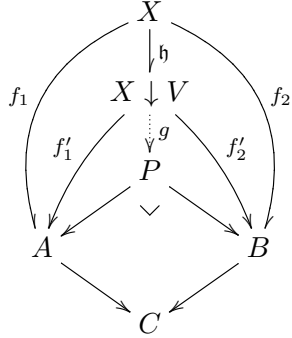
\[
\xymatrix@R=15pt@C=15pt{
&X
	\ar@^{->}^{\hid}[d]
	\ar@^{.>}@/_2pc/_{f_1}[dddl]
	\ar@{-}@/_2pc/[dddl]
	\ar@^{.>}@/^2pc/^{f_2}[dddr]
	\ar@{-}@/^2pc/[dddr]\\
&X\proj V
	\ar@/_/_{f'_1}[ddl]
	\ar@/^/^{f'_2}[ddr]
	\ar@{.>}^g[d]\\
&P
	\pb{270}
	\ar[dl]
	\ar[dr]\\
A	\ar[dr]&&
B	\ar[dl]\\
&C
}
\]
Take $f_1, f_2$ partial maps such that the outer square commutes. Necessarily, $f_1$ and $f_2$ are defined on the same subset
of events of $X$; call it $V$. By Lemma \ref{lem:partotup}, $\hid : X \pto X\proj V$ satisfies the partial-total universal
property. By the universal property of the pullback in $\ES$, there exists a unique $g : X\proj V \to P$ making the 
triangle commutes, yielding a factorization $g\circ \hid : X \pto P$. Uniqueness follows directly from the 
uniqueness of the pullback and of the partial-total universal property.
\end{proof}

Therefore, we can use the universal property of the interaction pullback to manipulate and compose hiding maps.
This allows us to state and prove the lemma below, which plays a similar role to
the \emph{zipping lemma} used in proving associativity of composition in sequential games -- hence the name.

\begin{lem}[Zipping lemma]
Let $\hid : S \pto S'$ be a hiding map making the following diagram commute:
\[
\xymatrix@R=10pt{
S       \ar@^{->}[r]^\hid
        \ar[d]_{\sigma}&
S'      \ar[d]^{\sigma'}\\
A\parallel B \parallel C
        \ar@^{->}[r]^-{A\parallel \bot \parallel C}&
A\parallel C
}
\]\medskip

\noindent Then, for $\rho : U \to C \parallel D$, the morphism $U\circledast \hid : U\circledast S \pto U \circledast S'$
defined using the universal property of $U \circledast S'$ via Lemma \ref{lem:partial_pb} is a hiding map.
\label{lem:zipping}
\end{lem}
\begin{proof} We show that $U\circledast \hid$ has a hiding witness.
A configuration of $U\circledast S'$ corresponds to configurations $x_{S'}\parallel x_D$ and $x_A \parallel x_U$
of the event structures as annotated, such that:
\begin{eqnarray*}
\sigma' x_{S'} &=& x_A \parallel x_C\\
\rho x_U &=& x_C \parallel x_D
\end{eqnarray*}
and such that the induced bijection between $x_{S'} \parallel x_D$ and $x_A \parallel x_U$ is secured.

From that, we consider $\wit_\hid(x_{S'}) \parallel x_D$ and $x_A \parallel x_B 
\parallel x_U$, where
$x_B$ is obtained by $\sigma(\wit_\hid(x_{S'})) = x_A \parallel x_B \parallel 
x_C$. By construction we have
$(\sigma \parallel D)(\wit_\hid(x_{S'}) \parallel x_D) = (A \parallel B 
\parallel \rho)(x_A \parallel x_B \parallel x_U)$.
The induced bijection is secured: 
a causal loop in it could not stay in (events projected to) $B$, as the causality
on the corresponding pairs is entirely determined by $S$. So, using that $\hid$ 
is partial rigid by
Proposition \ref{prop:hiding_carac} it would induce a causal loop in the original bijection, that was 
supposed secured. All the additional properties to check follow by construction.
\end{proof}

At this point, we can define the associator. Recall that for $\sigma : S \to A \parallel B$, $\tau : T \to B \parallel C$
and $\rho : U \to C \parallel D$ we have the associator at the level of interactions:
\[
a_{\sigma, \tau, \rho} : (U\circledast T) \circledast S \to U \circledast (T\circledast S)
\]
By using the two lemmas above, we have two hiding maps:
\[
\xymatrix@R=0pt{
\hid_{\sigma, (\tau, \rho)} = (U\circledast T) \circledast S
	\ar@^{->}[r]^{~~~~~~\hid_{\tau, \rho}\circledast S}&
(U\odot T) \circledast S
	\ar@^{->}[r]^{\hid_{\sigma, \rho\odot \tau}}&
(U\odot T) \odot S\\
\hid_{(\sigma, \tau), \rho} = U \circledast (T\circledast S)
	\ar@^{-}[r]^{~~~~~~U\circledast \hid_{\sigma, \tau}}&
U\circledast (T\odot S)
	\ar@^{->}[r]^{\hid_{\tau \odot \sigma, \rho}}&
U \odot (T\odot S)
}
\]
From the definitions, it is easy to check that the following outer diagram commutes:
\[
\xymatrix@R=10pt@C=10pt{
U \circledast (T\circledast S)
	\ar[rr]^{a_{\sigma, \tau, \rho}}
	\ar@^{->}[d]_{\hid_{\sigma, (\tau, \rho)}}&&
(U \circledast T) \circledast S
	\ar@^{->}[d]^{\hid_{(\sigma, \tau), \rho}}\\
U\odot (T\odot S)
	\ar@{.>}[rr]^{\alpha_{\sigma, \tau, \rho}}
	\ar[dr]_{\rho \odot (\tau \odot \sigma)}&&
(U\odot T) \odot S
	\ar[dl]^{(\rho \odot \tau) \odot \sigma}\\
&A\parallel D
}
\]
So by the partial-total universal properties of $\hid_{(\sigma, \tau), \rho}$ and $\hid_{\sigma, (\tau, \rho)}$,
$a_{\sigma, \tau, \rho}$ induces a unique isomorphism $\alpha_{\sigma, \tau, \rho} : (U\odot T) \odot S \to U \odot (T\odot S)$
making the two sub-diagrams commute.

\paragraph{Naturality and coherence.} To conclude the associativity part of the bicategory construction, we need to check that
these isomorphisms are natural in $\sigma, \tau, \rho$ and satisfy Mac Lane's pentagon. In both cases, the proof consists in 
verifying it first for interactions (as we already did earlier from the pentagon), and deducing it for composition by
checking that the maps involved in the diagram for composition are canonically related to those for interaction, as above.
We skip the details, that can be recovered easily.

\subsection{Unitors}
\label{subsec:unitors}

The last ingredients of our bicategory are the two unitors. For any strategy $\sigma : S \to A^\perp \parallel B$, those are 
the two isomorphisms for cancellation of copycat:
\begin{eqnarray*}
\rho_\sigma &=& S \odot \CCC_A \to S\\
\lambda_\sigma &=& \CCC_B \odot S \to S
\end{eqnarray*}

We start by defining $\lambda_\sigma$ (and $\rho_\sigma$): their definition is not strictly speaking covered by the result
of Theorem \ref{thm:main_thm} which only dealt with closed compositions of a strategy $\sigma : S \to A$ with $\ccc_A$. However
the construction is very similar and will only be roughly sketched here.

\begin{lem}
Let $\sigma : S \to A^\perp \parallel B$. Then, there are order-isomorphisms:
\[
\begin{array}{l}
\Psi_r : \conf{S\circledast \CCC_A} \iso \{(x^l_A, x_S) \in \conf{A}\times \conf{S} \mid \sigma x_S = x^r_A \parallel x_B~\&~x^l_A \sqsupseteq_A x^r_A\}\\
\Psi_l : \conf{\CCC_B \circledast S} \iso \{(x_S, x^r_B)\in \conf{S}\times \conf{B}\mid \sigma x_S = x_A \parallel x^l_B~\&~x^r_B \sqsubseteq_B x^l_B\}
\end{array}
\]
where the right hand side sets are ordered by componentwise inclusion.
\label{lem:deadlock2}
\end{lem}
\begin{proof}
Straightforward adaptation of Lemma \ref{lemma:deadlock}.
\end{proof}

At this point, it is also worth mentioning that it follows from courtesy of $\sigma$ that in a situation like in
the lemma above, we actually have $x^l_B \subseteq^- x^r_B$. No positive events can be added by going from $x^r_B$
to $x^l_B$, as using courtesy one can show that those could not be below a visible events. That fact is not used
in our development, so we skip the detailed proof.

We jump to the definition of the unitors:

\begin{lem}
For any $\sigma : S \to A^\perp \parallel B$, there are isomorphisms of strategies:
\[
\rho_\sigma : S \odot \CCC_A \to S~~~~~~~~~~~~~~~~\lambda_\sigma : \CCC_B \odot S \to S
\]
which respectively,
\begin{itemize}
\item To any $x\in \conf{S \odot \CCC_A}$ with unique witness $ [x] = \Psi_l^{-1}(x^l_A, x_S) \in \conf{\CCC_A\circledast S}$ with
$\sigma x_S = x^r_A \parallel x_B$ and 
$x^l_A \sqsubseteq_{A^\perp} x^r_A$, $\rho_\sigma$ associates the unique $x'_S \sqsubseteq x_S$ 
such that $\sigma x'_S = x^l_A \parallel x_B$ given by the discrete fibration property of $\sigma$.
\item To any $x \in \conf{\CCC_B \odot S}$ with unique witness $[x] = \Psi_r^{-1}(x_S, x^r_B) \in \conf{S \circledast \CCC_B}$
with $\sigma x_S = x_A \parallel x^l_B$ and 
$x^r_B \sqsubseteq_B x^l_B$, $\lambda_\sigma$ associates the unique $x'_S \sqsubseteq x_S$ such that $\sigma x'_S = x_A \parallel x^r_B$.
\end{itemize}
\label{lem:unitors}
\end{lem}
\begin{proof}
Straightforward adaptation of \emph{(iii) $\Rightarrow$ (i)} in
the proof of Theorem \ref{thm:main_thm}.
\end{proof}

We now show that the unitors $\lambda_\sigma$, $\rho_\sigma$ are natural in 
$\sigma$. In fact, it
will be helpful later on to prove here a slightly more general property: that the unitors acts
naturally with respect to generalized morphisms between strategies, that change the base game as
well.
In order to state it, first note that the
construction $A \mapsto \CCC_A$ on esps can be easily extended into a functor:
\[
\CCC : \ESP \to \ESP
\]
Indeed, for $f : A \to B$ a map of esps, we have $f^\perp \parallel f : A^\perp \parallel A \to
B^\perp \parallel B$ (using the obvious functorial action of $(-)^\perp$ and $\parallel$ on $\ESP$).
But $A^\perp \parallel A$ and $B^\perp \parallel B$ are respectively the sets of events of $\CCC_A$
and $\CCC_B$; and it is a simple verification that we do have $\CCC_f = f^\perp \parallel f : \CCC_A
\to \CCC_B$. Functoriality of the construction is clear. Using that, we state and prove the
following:

\begin{lem}
Let $\sigma_1 : S_1 \to A_1^\perp \parallel B_1$, $\sigma_2 : S_2 \to A_2^\perp \parallel B_2$, 
and $f : S_1 \to S_2$, $h : A_1 \to A_2, h' : B_1 \to B_2$ such that the following diagram commutes:
\[
\xymatrix{
S_1	\ar[r]^f
	\ar[d]_{\sigma_1}&
S_2	\ar[d]^{\sigma_2}\\
A_1^\perp \parallel B_1
	\ar[r]^-{h^\perp \parallel h'}&
A_2^\perp \parallel B_2
}
\]
Then, the following two diagrams commute as well:
\[
\xymatrix{
&&A_1^\perp \parallel B_1
	\ar[dd]_{h^\perp \parallel h'}\\
\CCC_{B_1} \odot S_1
	\ar[urr]^{\ccc_{B_1} \odot \sigma_1}
	\ar[r]_-{\lambda_{\sigma_1}}
	\ar[dd]^{\CCC_{h'}\odot f}&
S_1	\ar[ur]_{\sigma_1}
	\ar[dd]^f\\
&&A_2^\perp \parallel B_2\\
\CCC_{B_2} \odot S_2
	\ar[urr]^(.4){\ccc_{B_2}\odot \sigma_2\quad}
	\ar[r]_-{\lambda_{\sigma_2}}&
S_2	\ar[ur]_{\sigma_2}
}
\xymatrix{
&&A_1^\perp \parallel B_1
        \ar[dd]_{h^\perp \parallel h'}\\
S_1 \odot \CCC_{A_1}
        \ar[urr]^{\sigma_1 \odot \ccc_{A_1}}
        \ar[r]_-{\rho_{\sigma_1}}
        \ar[dd]^{f\odot \CCC_h}&
S_1     \ar[ur]_{\sigma_1}
        \ar[dd]^f\\
&&A_2^\perp \parallel B_2\\
S_2 \odot \CCC_{A_2}
        \ar[urr]^(.4){\sigma_2 \odot \ccc_{A_2}\quad}
        \ar[r]_-{\rho_{\sigma_2}}&
S_2     \ar[ur]_{\sigma_2}
}
\]
In particular (when $h, h'$ are identities), $\lambda_\sigma$ and $\rho_\sigma$ are natural in
$\sigma$.
\end{lem}
\begin{proof}
Let us focus on the left hand side diagram, the other is symmetric. Of all the faces of the diagram,
the right hand side one is by hypothesis, the upper and lower are by definition of unitors in 
Lemma \ref{lem:unitors}, and the left hand side one is by Lemma \ref{lem:gen_odot_fonc}. It remains
to prove that the front face commutes.

Let $x \in \conf{\CCC_{B_1} \odot S_1}$, with unique witness $[x] = \Psi_r(x_{S_1}, x^r_{B_1})$, with
$\sigma_1 x_{S_1} = x_{A_1} \parallel x^l_{B_1}$ and $x^r_{B_1} \sqsubseteq x^l_{B_1}$. The left
unitor $\lambda_{\sigma_1}$ sends $x$ to the unique $x'_{S_1} \sqsubseteq x_{S_1}$ such that $\sigma
x'_{S_1} = x_{A_1} \parallel x^r_{B_1}$, whereas $\CCC_{h'} \odot f$ by definition sends it to
$(\CCC_{h'} \odot f)(x)$ with unique witness $\Psi_r(f(x_{S_1}), h'(x^r_{B_1}))$. But then,
$f(x'_{S_1}) \sqsubseteq f(x_{S_1})$ is such that $\sigma_2(f(x'_S)) = h(x_{A_1}) \parallel
h'(x^r_{B_1})$, and the unique such (by uniqueness of the discrete fibration property). Therefore,
$\lambda_{\sigma_2}((\CCC_{h'} \odot f)(x)) = f(x'_{S_1})$.
\end{proof}

And finally, using the description of their action we verify the coherence law for unitors.

\begin{lem}
For $\sigma : S \to A^\perp \parallel B$ and $\tau : T \to B^\perp \parallel C$,
the following diagram commutes.
\[
\xymatrix{
(T\odot \CCC_B)\odot S
        \ar[rr]^{\alpha_{\sigma, \ccc_B, \tau}}
        \ar[dr]_{\rho_{\tau}\odot S}&&
T\odot (\CCC_B \odot S)
        \ar[dl]^{T\odot \lambda_\sigma}\\
&T\odot S
}
\]
\end{lem}
\begin{proof}
Let $x \in \conf{(T\odot \CCC_B)\odot S}$. Necessarily, it has a witness $\wit(x) \in \conf{(T\circledast \CCC_B)\circledast S}$.
By characterisation of pullbacks, it corresponds to three configurations $x_S \parallel x^r_B \parallel x_C$,
$x_A \parallel x_B^l \parallel x_B^r \parallel x_C$, and $x_A \parallel x_B^l \parallel x_T$ such that
$\sigma x_S = x_A \parallel x_B^l$, $x_B^r \sqsubseteq x_B^l$ (regarded as configurations of $B$), and
$\tau x_T = x_B^r \parallel x_C$. Moreover, the induced order on triples is secured, and its maximal
elements are visible. But this implies that actually $x_B^l = x_B^r$ -- it is easy to show that if
(non-visible) $b \in x_B^l$ is not in $x_B^r$, then it cannot be below a visible event. From that
it follows that both paths alongside the triangle above map $x$ to (the configuration of $T\odot S$
represented by) $x_S \parallel x_C$ and $x_A \parallel x_T$.
\end{proof}

We have finished the proof that $\CG$ is a bicategory.

\section{A compact-closed (bi)category}
\label{sec:compact}
In this section, we show that similarly to Joyal's category of Conway 
games, our bicategory of concurrent games has a compact closed structure, a structure
that is central in the applications of our framework to game semantics of programming
languages. 

Recall that a compact closed category is a symmetric monoidal category, where each object 
$A$ has a \emph{dual} $A^*$, which is related to $A$ via two morphisms:

\[
\eta_A : 1 \profto A^* \tensor A \hspace{80pt} \epsilon_A : A\tensor A^* \profto 1
\]
where $1$ is the unit of the tensor (in our concrete case it is the empty game). These
morphisms have to obey two laws that are best represented in the language of string diagrams:

\begin{center}
\scalebox{.6}{
\begin{tikzpicture}[]
\useasboundingbox (-0.5,-0.5) rectangle (17.5,4.5);
\draw (5.00,4.00) -- (4.76,3.97) -- (4.52,3.94) -- (4.28,3.91) -- (4.05,3.88) -- (3.83,3.85) -- (3.61,3.82) -- (3.39,3.78) -- (3.19,3.74) -- (3.00,3.70) -- (2.82,3.66) -- (2.66,3.61) -- (2.51,3.56) -- (2.37,3.51) -- (2.25,3.45) -- (2.16,3.39) -- (2.08,3.32) -- (2.02,3.25) -- (1.99,3.17) -- (1.98,3.09) -- (2.00,3.00);
\draw (2.00,3.00) -- (2.04,2.91) -- (2.10,2.82) -- (2.17,2.73) -- (2.26,2.63) -- (2.36,2.53) -- (2.48,2.43) -- (2.60,2.32) -- (2.73,2.22) -- (2.86,2.11) -- (3.00,2.00) -- (3.14,1.89) -- (3.27,1.78) -- (3.40,1.68) -- (3.52,1.57) -- (3.64,1.47) -- (3.74,1.37) -- (3.83,1.27) -- (3.90,1.18) -- (3.96,1.09) -- (4.00,1.00);
\draw (4.00,1.00) -- (4.02,0.91) -- (4.01,0.83) -- (3.98,0.75) -- (3.92,0.68) -- (3.84,0.61) -- (3.75,0.55) -- (3.63,0.49) -- (3.49,0.44) -- (3.34,0.39) -- (3.18,0.34) -- (3.00,0.30) -- (2.81,0.26) -- (2.61,0.22) -- (2.39,0.18) -- (2.17,0.15) -- (1.95,0.12) -- (1.72,0.09) -- (1.48,0.06) -- (1.24,0.03) -- (1.00,0.00);
\draw (8.00,4.00) -- (8.24,3.97) -- (8.48,3.94) -- (8.72,3.91) -- (8.95,3.88) -- (9.17,3.85) -- (9.39,3.82) -- (9.61,3.78) -- (9.81,3.74) -- (10.00,3.70) -- (10.18,3.66) -- (10.34,3.61) -- (10.49,3.56) -- (10.63,3.51) -- (10.75,3.45) -- (10.84,3.39) -- (10.92,3.32) -- (10.98,3.25) -- (11.01,3.17) -- (11.02,3.09) -- (11.00,3.00);
\draw (11.00,3.00) -- (10.96,2.91) -- (10.90,2.82) -- (10.83,2.73) -- (10.74,2.63) -- (10.64,2.53) -- (10.52,2.43) -- (10.40,2.32) -- (10.27,2.22) -- (10.14,2.11) -- (10.00,2.00) -- (9.86,1.89) -- (9.73,1.78) -- (9.60,1.68) -- (9.48,1.57) -- (9.36,1.47) -- (9.26,1.37) -- (9.17,1.27) -- (9.10,1.18) -- (9.04,1.09) -- (9.00,1.00);
\draw (9.00,1.00) -- (8.98,0.91) -- (8.99,0.83) -- (9.02,0.75) -- (9.08,0.68) -- (9.16,0.61) -- (9.25,0.55) -- (9.37,0.49) -- (9.51,0.44) -- (9.66,0.39) -- (9.82,0.34) -- (10.00,0.30) -- (10.19,0.26) -- (10.39,0.22) -- (10.61,0.18) -- (10.83,0.15) -- (11.05,0.12) -- (11.28,0.09) -- (11.52,0.06) -- (11.76,0.03) -- (12.00,0.00);
\draw (15.00,2.00) -- (15.05,2.00) -- (15.10,2.00) -- (15.15,2.00) -- (15.20,2.00) -- (15.25,2.00) -- (15.30,2.00) -- (15.35,2.00) -- (15.40,2.00) -- (15.45,2.00) -- (15.50,2.00) -- (15.55,2.00) -- (15.60,2.00) -- (15.65,2.00) -- (15.70,2.00) -- (15.75,2.00) -- (15.80,2.00) -- (15.85,2.00) -- (15.90,2.00) -- (15.95,2.00) -- (16.00,2.00);
\draw (16.00,2.00) -- (16.05,2.00) -- (16.10,2.00) -- (16.15,2.00) -- (16.20,2.00) -- (16.25,2.00) -- (16.30,2.00) -- (16.35,2.00) -- (16.40,2.00) -- (16.45,2.00) -- (16.50,2.00) -- (16.55,2.00) -- (16.60,2.00) -- (16.65,2.00) -- (16.70,2.00) -- (16.75,2.00) -- (16.80,2.00) -- (16.85,2.00) -- (16.90,2.00) -- (16.95,2.00) -- (17.00,2.00);
\filldraw[fill=white] (2.00,3.00) ellipse (0.80cm and 0.50cm);
\filldraw[fill=white] (11.00,3.00) ellipse (0.80cm and 0.50cm);
\filldraw[fill=white] (4.00,1.00) ellipse (0.80cm and 0.50cm);
\filldraw[fill=white] (9.00,1.00) ellipse (0.80cm and 0.50cm);
\draw (5.50,4.00) node{$A$};
\draw (7.50,4.00) node{$A$};
\draw (2.00,3.00) node{$\eta_{A^*}$};
\draw (11.00,3.00) node{$\epsilon_A$};
\draw (6.50,2.00) node{$=$};
\draw (13.50,2.00) node{$=$};
\draw (14.50,2.00) node{$A$};
\draw (17.50,2.00) node{$A$};
\draw (4.00,1.00) node{$\epsilon_{A^*}$};
\draw (9.00,1.00) node{$\eta_A$};
\draw (0.50,0.00) node{$A$};
\draw (12.50,0.00) node{$A$};
\end{tikzpicture}
}
\end{center}

Compact closed categories play an important role in the background in semantics: the equations of compact
closed categories are mirrored, \emph{e.g.} in the reduction rules of proof nets and in the 
adjunction laws ($\beta$ and $\eta$-conversion) of cartesian closed or symmetric 
closed monoidal categories.
In fact, any compact closed category is symmetric closed monoidal (more 
precisely, $*$-autonomous, and a model of MLL \cite{girard1987linear}):
setting $A\multimap B = A^* \tensor B$, we have the adjunction $A\tensor - \dashv A \multimap -$.
In short, compact closed categories form the backbone of an equational presentation of the dynamics
of linear higher-order computation.

But unlike Conway games, $\CG$ is a \emph{bicategory}.
In fact, we believe that it gives an example of a \emph{compact closed bicategory}, as defined by Kelly \cite{Kelly} and detailed by Stay \cite{1301.1053}.
However, the precise definition of a compact closed bicategory is 
rather intimidating. It might be possible to deduce the bicategorical
compact closed structure of $\CG$ from that of the bicategory of profunctors \cite{1301.1053}, but
{we do not do so here. Although in subsequent work we occasionally rely on the algebraic
laws for composition pre-quotient, the
literature and body of work that we need to connect to when setting up our game semantics (for
instance, models of linear logic \cite{panorama}) does not exploit this bicategorical
structure.}
So, we only check that the \emph{quotiented category} is compact closed.

By abuse of notations, from now on we will use the same notation $\CG$ for the quotiented category instead of the bicategory.
Regarded as a category, $\CG$ has esps as objects, and as morphisms strategies $\sigma : S \to A^\perp \parallel B$ up to
isomorphism. In the rest of this section, we check the components of a compact closed category.

\subsection{The bifunctor}

First, we define a bifunctor $\tensor : \CG^2 \to \CG$. On objects, $A\tensor B$ is simply defined as $A \parallel B$. On morphisms,
for $\sigma_1 : S_1 \to A_1^\perp \parallel B_1$ and $\sigma_2 : S_2 \to A_2^\perp \parallel B_2$, we define
\[
\sigma_1 \tensor \sigma_2 = 
\xymatrix@C=50pt{
S_1 \parallel S_2
        \ar[r]^-{\sigma_1 \parallel \sigma_2~~~~~~~~~}&
(A_1^\perp \parallel B_1) \parallel (A_2^\perp \parallel B_2)
        \ar[r]^-{\gamma_{A_1^\perp, B_1, A_2^\perp, B_2}}&
(A_1 \parallel A_2)^\perp \parallel (B_1 \parallel B_2)
}
\]
where $\gamma_{A, B, C, D} : (A\parallel B) \parallel (C \parallel D) \to (A\parallel C) \parallel (B \parallel D)$ is the
obvious isomorphism of esps. We show that this operation is a bifunctor. Firstly, it preserves the identity.

\begin{prop}
For any esp $A$, we have
\[
\ccc_{A\tensor B} \iso \ccc_A \tensor \ccc_B
\]
\end{prop}
\begin{proof}
We have the isomorphism
\[
\gamma_{A^\perp, B^\perp, A, B} : (A^\perp \parallel B^\perp) \parallel (A \parallel B) \to (A^\perp \parallel A) \parallel (B^\perp \parallel B)
\]
which can also be typed as $\gamma_{A^\perp ,B^\perp, A, B} : \CCC_{A\tensor B} \to \CCC_A \parallel \CCC_B$, which obviously commutes with
the projections to the game.
\end{proof}

Secondly, it preserves composition.

\begin{prop}
Let:
\[
\begin{array}{rcrclcrcrcl}
\sigma_1 &:& S_1 &\to& A_1^\perp \parallel B_1 &~~~~~~~~~~& \tau_1 &:& T_1 &\to& B_1^\perp \parallel C_1\\
\sigma_2 &:& S_2 &\to& A_2^\perp \parallel B_2 &~~~~~~~~~~& \tau_2 &:& T_2 &\to& B_2^\perp \parallel C_2
\end{array}
\]
Then,
\[
(\tau_1 \odot \sigma_1) \tensor (\tau_2 \odot \sigma_2) \iso (\tau_1 \tensor \tau_2) \odot (\sigma_1 \tensor \sigma_2)
\]
\end{prop}
\begin{proof}
We start by proving it for interactions. As the parallel composition of pullback squares is a pullback square, we have two
pullbacks related by isomorphisms:
\[
\xymatrix@C=0pt@R=15pt{
&(T_1 \circledast S_1) \parallel (T_2 \circledast S_2)
        \ar[dl]
        \ar[dr]
        \pb{270}\\
(S_1 \parallel C_1) \parallel (S_2 \parallel C_2)
        \ar[dr]^(.55){\qquad(\sigma \parallel C_1) \parallel (\sigma_2 \parallel C_2)}
        \ar@{.>}@/^4pc/[dddd]^{\gamma_{S_1, C_1, S_2, C_2}}&&
(A_1 \parallel T_1) \parallel (A_2 \parallel T_2)
        \ar[dl]^(.4){\qquad(A_1 \parallel \tau_1) \parallel (A_2 \parallel \tau_2)}
        \ar@{.>}@/^4pc/[dddd]_{\gamma_{A_1, T_2, A_2, T_2}}\\
& (A_1 \parallel B_1 \parallel C_1) \parallel (A_2 \parallel B_2 \parallel C_2)
        \ar@{.>}@/^4pc/[dddd]^{\delta}\\\\
&(T_1 \parallel T_2) \circledast (S_1 \parallel S_2)
        \ar[dl]
        \ar[dr]
        \pb{270}\\
(S_1 \parallel S_2) \parallel (C_1 \parallel C_2)
        \ar[dr]^(.55){\qquad(\sigma_1 \tensor \sigma_2) \parallel (C_1 \parallel C_2)}&&
(A_1 \parallel A_2) \parallel (T_1 \parallel T_2)
        \ar[dl]^(.4){\qquad(A_1 \parallel A_2) \parallel (\tau_1 \tensor \tau_2)}\\
&(A_1 \parallel A_2) \parallel (B_1 \parallel B_2) \parallel (C_1 \parallel C_2)
}
\]
where $\delta$ is the obvious map. By universal property of the pullback that gives an isomorphism:
\[
\gamma' : (T_1 \circledast S_1) \parallel (T_2 \circledast S_2) \iso (T_1 \parallel T_2) \circledast (S_1 \parallel S_2)
\]
which commutes (up to $\gamma_{A_1, C_1, A_2, C_2}$) with the hiding maps $\hid_{\sigma_1, \tau_1} \parallel \hid_{\sigma_2, \tau_2}$ and
$\hid_{\sigma_1 \tensor  \sigma_2 , \tau_1\tensor \tau_2}$, so using 
Proposition \ref{prop:hiding_carac} and the easy fact that maps with hiding 
witnesses are stable by parallel composition, it follows that $\gamma'$ corresponds to a unique isomorphism:
\[
(T_1 \odot S_1) \parallel (T_2 \odot S_2) \iso (T_1 \parallel T_2) \odot (S_1 \parallel S_2)
\]
between strategies $(\tau_1 \odot \sigma_1) \tensor (\tau_2 \odot \sigma_2)$ and $(\tau_1 \tensor \tau_2) \odot (\sigma_1 \tensor \sigma_2)$.
\end{proof}

\subsection{Lifting and symmetric monoidal structure of $\CG$}

The strategies serving as structural morphisms for the symmetric closed 
monoidal structure are very simple variants of copycat $\ccc_A : A \profto A$. In 
order to construct the symmetric monoidal structure
of $\CG$, we describe a systematic way of generating such morphisms from more elementary maps
of esps.

\begin{defi}
Let $f : A \to B$ be a receptive, courteous map of esps\footnote{This means that, technically, $f$ is a
strategy on $B$ -- though we are not thinking of it that way.}. Then, the map:
\[
\begin{array}{rcrcl}
\lift{f} &:& \CCC_A &\to & A^\perp \parallel B\\
&& a &\mapsto& (A^\perp \parallel f)\circ \ccc_A(a)
\end{array}
\]
is a strategy called the \textbf{lifting} of $f$. Likewise, if $f: B^\perp \to A^\perp$ is receptive and courteous,
we define its \textbf{co-lifting}:
\[
\begin{array}{rcrcl}
\lift{f}^\perp &:& \CCC_B &\to& A^\perp \parallel B\\
&& c &\mapsto& (f \parallel B)\circ \ccc_B(c)
\end{array}
\]
\end{defi}

The fact that they are strategies follows from the fact that courteous receptive maps are stable under composition.

The following key lemma links composition of strategies with lifted maps and composition of the corresponding
maps in $\ES$.

\begin{lem}
Let $f: B \to C$ be a receptive courteous map of esps, and $\sigma : S \to A^\perp \parallel B$ be a strategy. Then,
the unitor $\lambda_\sigma : \CCC_B \odot S \to S$ is an isomorphism between
\[
\begin{array}{rcrcl}
\lift{f}\odot \sigma &:& \CCC_B \odot S &\to& A^\perp \parallel C\\
(A^\perp \parallel f)\circ \sigma &:& S &\to& A^\perp \parallel C
\end{array}
\]
Likewise, for $f: B^\perp \to A^\perp$ receptive courteous and $\sigma : S \to B^\perp \parallel C$ a
strategy, $\rho_\sigma$ is an isomorphism between:
\[
\begin{array}{rcrcl}
\sigma \odot \lift{f}^\perp&:& S\odot \CCC_B &\to& A^\perp \parallel C\\
(f \parallel C)\circ \sigma &:& S &\to& A^\perp \parallel C
\end{array}
\]
\label{lem:lifting_key}
\end{lem}
\begin{proof}
By definition, the following two diagrams commute:
\[
\xymatrix{
S	\ar[r]^{S}
	\ar[d]^{\sigma}&
S	\ar[d]^{\sigma}\\
A^\perp \parallel B
	\ar[r]^{A^\perp \parallel B}&
A^\perp \parallel B
}
~~~~~~~~~~
\xymatrix{
\CCC_B	\ar[r]^{\CCC_B}
	\ar[d]^{\ccc_B}&
\CCC_B	\ar[d]^{\lift{f}}\\
B^\perp \parallel B
	\ar[r]^{B^\perp \parallel f}&
B^\perp \parallel C
}
\]
Therefore, by Lemma \ref{lem:gen_odot_fonc}, it follows that the following diagram commutes:
\[
\xymatrix{
\CCC_B\odot S
	\ar[r]^{\CCC_B\odot S}
	\ar[d]^{\ccc_B \odot \sigma}&
\CCC_B\odot S
	\ar[d]^{\lift{f}\odot \sigma}\\
A^\perp \parallel B
	\ar[r]^{A^\perp \parallel f}&
A^\perp \parallel C
}
\]
Combined with the isomorphism $\ccc_B \odot \sigma \iso \sigma$, this concludes the proof. The other case is symmetric.
\end{proof}

From the lemma above it immediately follows that lifting is functorial:

\begin{lem}
Let $f : A \to B$ and $g : B \to C$ be receptive courteous maps, then we have an isomorphism:
\[
\lift{g}\odot \lift{f} \iso \lift{g\circ f}
\]
\label{lem:lifting_functorial}
\end{lem}
\begin{proof}
Immediate consequence of Lemma \ref{lem:lifting_key}.
\end{proof}

Using this, we can lift the symmetric closed monoidal structure of $\ES$ to 
$\CG$. In particular,
there are natural isomorphisms in $\ES$ which are componentwise receptive and courteous, and so are their inverses.

\[
\begin{array}{rcrcl}
\rho_A &:& A\parallel 1 &\to & A\\
\lambda_A &:& 1\parallel A &\to& A\\
s_{A, B} &:& A\parallel B &\to& B \parallel A\\
\alpha_{A, B, C} &:& (A \parallel B) \parallel C &\to& A \parallel (B\parallel C)
\end{array}
\]
(the reuse of symbols from Section \ref{sec:bicategory} for these structural morphisms should not cause any confusion). These
isomorphisms can then be lifted to strategies:

\[
\begin{array}{rcrcl}
\lift{\rho_{A}} &:& A \parallel 1 &\profto& A\\
\lift{\lambda_{A}} &:& 1\parallel A &\profto& A\\
\lift{s_{A, B}}  &:& A\parallel B &\profto& B \parallel A\\
\lift{\alpha_{A, B, C}} &:& (A \parallel B) \parallel C &\profto& A \parallel (B\parallel C)
\end{array}
\]
which inherit from $\ES$ all the coherence laws of the symmetric monoidal structure by Lemma \ref{lem:lifting_functorial}.
It remains to prove that these families are natural.

%
%
%
%
%
%
%

\begin{lem}
The families $\rho_A, \lambda_A, s_{A, B}, \alpha_{A, B, C}$ are natural in all their components.
\end{lem}
\begin{proof}
A direct verification. For illustration, we detail the naturality of $s_{A, B}$.

Let $\sigma : S \to A_1^\perp \parallel A_2$, and $\tau : T \to B_1^\perp \parallel B_2$. We need to check:
\[
\lift{s_{A_2, B_2}} \odot (\sigma \tensor \tau) \iso (\tau \tensor \sigma) \odot \lift{s_{A_1, B_1}}
\]
But there is an obvious isomorphism $\lift{s_{A_1, B_1}} \iso \lift{s_{A_1^\perp, B_1^\perp}^{-1}}^\perp$. So by
both parts of Lemma \ref{lem:lifting_key}, this amounts to finding an isomorphism between the two maps:
\[
\xymatrix@R=0pt@C=165pt{
S\parallel T
        \ar[r]^-{((A_1 \parallel B_1)^\perp \parallel s_{A_2, B_2})\circ \gamma_{A_1^\perp, A_2, B_1^\perp, B_2} \circ (\sigma \parallel \tau)}&
(A_1\parallel B_1)^\perp \parallel (B_2 \parallel A_2)\\
T\parallel S
        \ar[r]^-{(s_{A_1^\perp, B_1^\perp}^{-1} \parallel (B_2 \parallel A_2))\circ \gamma_{B_1^\perp, B_2, A_1^\perp, A_2} \circ (\tau \parallel \sigma)}&
(A_1\parallel B_1)^\perp \parallel (B_2 \parallel A_2)
}
\]
and it is a simple verification to check that $s_{S, T}$ does the trick.
\end{proof}

This concludes the symmetric monoidal structure of $\CG$.

\subsection{Compact closed structure}

The dual of a game $A$ is simply defined as $A^\perp$. We have two strategies:
\[
\begin{array}{rcrcl}
\eta_A &:& \CCC_A &\to& 1^\perp \parallel (A^\perp \parallel A)\\
\epsilon_A &:& \CCC_A &\to& (A\parallel A^\perp)^\perp \parallel 1
\end{array}
\]
defined in the obvious way. We have:

\begin{prop}
The strategies $\eta_A : 1 \profto A^\perp \parallel A$ and $\epsilon_A : A \parallel A^\perp \profto 1$ satisfy the laws for a
compact closed category.
\end{prop}
\begin{proof}
We need to check the two equations of duals in compact closed categories:
\begin{eqnarray*}
\ccc_A &\iso& \lift{\lambda_{A}} \odot (\epsilon_A \tensor \ccc_A) \odot \lift{\alpha_{A, A^\perp,
A}^{-1}} \odot (\ccc_A \tensor \eta_A) \odot \lift{\rho_{A}}^{-1}\\
\ccc_{A^\perp} &\iso& \lift{\rho_{A^\perp}} \odot (\ccc_{A^\perp} \tensor \epsilon_A) \odot
\lift{\alpha_{A^\perp, A, A^\perp}} \odot (\eta_A \tensor \ccc_{A^\perp}) \odot
\lift{\lambda_{A^\perp}}^{-1}
\end{eqnarray*}
These two isomorphisms are symmetric; we only check the first. Let us write $\sigma : S \to A^\perp \parallel A$ for the resulting composition, and
\[
\xi : U \to A \parallel (A \parallel 1) \parallel (A \parallel (A\parallel A)) \parallel ((A\parallel A) \parallel A) \parallel (1 \parallel A) \parallel A
\]
for the corresponding $5$-ary composition. By Lemma \ref{lem:zipping}, there is a hiding map
$\hid : U \pto S$, commuting with the projection to the game. From the characterisation of configurations
of pullbacks, and after eliminating redundancies, configurations of $U$ correspond to the data of a configuration in each component $A$ above,
satisfying the following constraints:
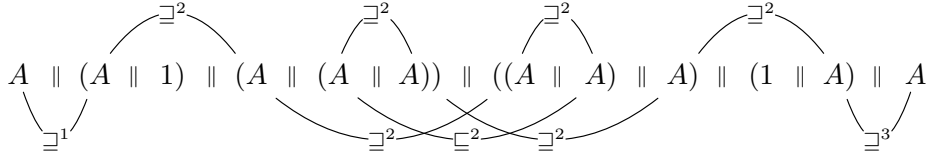
\[
\xymatrix@C=12pt{
A       \ar@{}[r]|\parallel
        \ar@{-}@/_2pc/[r]|{\sqsupseteq^1}&
(A      \ar@{}[r]|\parallel
        \ar@{-}@/^2pc/[rr]|{\sqsupseteq^2}&
1)      \ar@{}[r]|\parallel&
(A      \ar@{}[r]|\parallel
        \ar@{}@/_2pc/[rrr]|{\sqsupseteq^2}&
(A      \ar@{}[r]|\parallel
        \ar@{}@/^2pc/[r]|{\sqsupseteq^2}&
A))     \ar@{}[r]|\parallel
        \ar@{}@/_2pc/[rrr]|{\sqsupseteq^2}&
((A     \ar@{}[r]|\parallel
        \ar@{}@/^2pc/[r]|{\sqsupseteq^2}&
A)      \ar@{}@/^2pc/[lll]|{\sqsubseteq^2}
        \ar@{}[r]|\parallel&
A)      \ar@{}[r]|\parallel
        \ar@{}@/^2pc/[rr]|{\sqsupseteq^2}&
(1      \ar@{}[r]|\parallel&
A)      \ar@{}[r]|\parallel
        \ar@{}@/_2pc/[r]|{\sqsupseteq^3}&
A
}
\]
where, moreover, configurations whose maximal events are visible (and so correspond to configurations of $S$) are those where the $\sqsubseteq^1$ are replaced by
$\supseteq^+$, the $\sqsubseteq^2$ are replaced by equalities and the $\sqsubseteq^3$ are replaced by $\subseteq^-$.
Such configurations exactly correspond to those of $\CCC_A$.
\end{proof}

This
concludes the proof that $\CG$ is a compact closed category.

\section{Conclusions}
\label{sec:conclusion}
In this paper, we gave a detailed exposition of the results of \cite{lics11}, along with
some extensions.
We presented a notion of concurrent games based on event structures, which is
a concurrent analogue of Joyal's compact closed category of Conway games \cite{JoyalGazette}. 

We first defined pre-strategies, as certain event structures describing the evolution
of concurrent processes on an interface presented as a game. We defined strategies as
those pre-strategies stable under the action of an asynchronous forwarder, presented as
the copycat strategy. Finally, we proved that composition of strategies obeys the laws of a 
bicategory, and that just as Joyal's, the corresponding quotient category is compact
closed. As {explained} in \cite{DBLP:conf/fossacs/Winskel13}, it relates to the compact closed
bicategory of profunctors via a lax functor.

\paragraph{Further work.}
The developments presented in this paper are just the beginning of the story.
Since the appearance of \cite{lics11}, this framework has been used as a basis
for a number of extensions. In \cite{DBLP:conf/lics/ClairambaultGW12}, games were equipped with winning conditions. It was
proved that winning strategies also form a bicategory, and that just as in the sequential
case, well-founded games that satisfy a further condition called \emph{race-freeness} are
determined. This was later extended to all Borel winning conditions \cite{DBLP:journals/jcss/GutierrezW14}, provided in addition that concurrency
is bounded in the game. Winning conditions
were also generalized to a quantitative notion of payoff in \cite{DBLP:journals/entcs/ClairambaultW13}, and a value theorem was proved.
As witnessed by these determinacy results, and despite concurrency, our games remain total
information games (unlike \emph{e.g.} \cite{cgtcs07}). We investigated in \cite{Kozen, DBLP:conf/birthday/ClairambaultGW13} an extension to partial information
games, where determinacy is lost. The fourth author also extended the setting to probabilistic and 
quantum strategies \cite{DBLP:journals/entcs/Winskel13}.

In our basic setting, games are affine: each event can occur at most \emph{once}. It is 
key for many applications (most notably to semantics) that one 
allows the replication of events,
in such a way that distinct copies are indistinguishable from each other. To this effect, 
we equipped games with a notion of symmetry expressing indistinguishability of events and configurations. Strategies
then have to respect this additional structure, by treating symmetric configurations uniformly. This can be done in two ways: the first option
is to saturate strategies by forcing them to play non-deterministically all symmetric events.
In \cite{lics14}, we developed a bicategory of saturated strategies on games with symmetry, using it to
allow replication and construct analogues of AJM \cite{DBLP:journals/iandc/AbramskyJM00} and HO \cite{DBLP:journals/iandc/HylandO00} games. 
In \cite{DBLP:conf/lics/CastellanCW15} we developed a second
option, and showed that with some minimality assumption on strategies one could obtain a bicategory
of uniform strategies while avoiding saturation and the addition of 
redundant non-deterministic choices.
We showed that this gave a cartesian closed category, supporting an intensionally fully abstract
model of PCF where independent sub-computations are performed in parallel.

\paragraph{Perspectives.} 
There is a lot of ongoing work on the topic of concurrent games on event structures. On the fundamental
side, we have looked 
for a generalization of the basic setting presented here that accommodates better events with \emph{disjunctive causality},
\emph{i.e.} that can occur for several distinct yet compatible reasons~\cite{deVismeWinskel-archived}. On the semantic side, we have 
several research directions. To
name a few, we want to represent non-interference as determinism in concurrent languages; to enrich strategies to keep information about possible
local deadlocks or divergences; to investigate further strategies from the point of view of concurrent processes~\cite{CHLW}; and to mix symmetry with probabilities in order to build a denotational model combining
probabilities, non-determinism and concurrency.

But beyond semantics, our concurrent games give a powerful and precise description of the evolution of concurrent
processes. We wish to extend this basic framework in order to set a standard for a concurrent notion of games and
strategies. We hope this framework will then be a relevant and useful tool for 
various purposes, from handling algorithmic issues in concurrency to 
investigating its logical properties.\newpage

\bibliographystyle{alpha}
\bibliography{main}

\end{document}